\documentclass[11pt,a4paper]{article}

\usepackage[utf8x]{inputenc}

\usepackage{color}
\usepackage{subfigure}
\usepackage{url}
\usepackage{graphicx}
\usepackage{amsmath}
\usepackage{amsthm}
\usepackage{bm}
\usepackage[plainpages=false]{hyperref}

\usepackage{titlesec}
\titleformat{\section}
	{\normalfont\Large\bfseries\filcenter}{\thesection.}{1 ex}{}
\titleformat{\subsection}
	{\normalfont\normalsize\bfseries}{\thesubsection.}{1 ex}{}
\titleformat{\subsubsection}[runin]
	{\normalfont\normalsize\bfseries\filcenter}{\thesubsubsection.}{1 ex}{}

\usepackage[charter]{mathdesign}
\usepackage[mathcal]{eucal}

\usepackage[margin=1.2 in]{geometry}

\usepackage[square,numbers,sort&compress]{natbib}

\newcommand{\emailaddr}[1]{\href{mailto:#1}{\texttt{#1}}}

\usepackage{thmtools,thm-restate}
\declaretheoremstyle[headpunct={ --- },headfont=\normalfont\itshape]{myremark}
\declaretheoremstyle[bodyfont=\normalfont]{mydefinition}
\declaretheorem[name=Theorem]{Thm}
\declaretheorem[within=section,name=Lemma]{Lem}
\declaretheorem[sibling=Lem,name=Proposition]{Prop}
\declaretheorem[sibling=Lem,name=Corollary]{Cor}

\declaretheorem[style=myremark,sibling=Lem,name=Remark]{Rem}
\declaretheorem[style=mydefinition,sibling=Lem,name=Definition]{Def}
\declaretheorem[style=mydefinition,sibling=Lem,name=Notation]{Not}

\declaretheorem[style=mydefinition,sibling=Lem,name=Example]{Ex}

\declaretheorem[style=mydefinition,sibling=Lem,name=Problem]{Prob}

\usepackage{enumerate}
\usepackage{fancyhdr}
\usepackage{verbatim}
\usepackage{array}
\usepackage{fullpage}
\usepackage{enumitem}
\usepackage[vcentermath, enableskew]{youngtab}
\usepackage{multicol}
\usepackage{longtable}
\usepackage[small,bf,singlelinecheck=off]{caption}
\usepackage{bbm}

\newcommand{\Tr}{\operatorname{Tr}}

\newcommand{\card}{\#}

\newcommand{\calA}{\mathcal{A}}
\newcommand{\calB}{\mathcal{B}}
\newcommand{\calC}{\mathcal{C}}

\newcommand{\calP}{\mathcal{P}}
\newcommand{\calQ}{\mathcal{Q}}

\newcommand{\calS}{\mathcal{S}}

\newcommand{\calV}{\mathcal{V}}
\newcommand{\calX}{\mathcal{X}}
\newcommand{\calY}{\mathcal{Y}}
\newcommand{\calZ}{\mathcal{Z}}

\newcommand{\fraks}{\mathfrak{s}}

\newcommand{\RR}{\ensuremath{\mathbb{R}}}

\newcommand{\CC}{\ensuremath{\mathbb{C}}}
\newcommand{\NN}{\ensuremath{\mathbb{N}}}
\newcommand{\KK}{\ensuremath{\mathbb{K}}}

\DeclareMathOperator*{\Id}{I}
\DeclareMathOperator*{\Van}{V}

\makeatletter
\providecommand*{\diff}%
        {\@ifnextchar^{\DIfF}{\DIfF^{}}}
\def\DIfF^#1{%
        \mathop{\mathrm{\mathstrut d}}%
                \nolimits^{#1}\gobblespace
}
\def\gobblespace{%
        \futurelet\diffarg\opspace}
\def\opspace{%
        \let\DiffSpace\!%
        \ifx\diffarg(%
                \let\DiffSpace\relax
        \else
                \ifx\diffarg\[%
                        \let\DiffSpace\relax
                \else
                        \ifx\diffarg\{%
                                \let\DiffSpace\relax
                        \fi\fi\fi\DiffSpace}
\makeatother

\begin{document}
\title{The Algebraic Approach to Phase Retrieval, and\\ 
Explicit Inversion at the Identifiability Threshold}
 \author{Franz J. Kiraly\thanks{Department of Statistical Science, University College London. \url{f.kiraly@ucl.ac.uk}}
\and
Martin Ehler\thanks{University of Vienna. \url{martin.ehler@univie.ac.at}}}
\author{
	\href{http://www.ucl.ac.uk/statistics/people/franz-kiraly}{Franz J. Király} \thanks{Department of Statistical Science, University College London; and MFO; \emailaddr{f.kiraly@ucl.ac.uk}}
	\and \href{http://homepage.univie.ac.at/martin.ehler/}{Martin Ehler}
		\thanks{Department of Mathematics, University of Vienna, \emailaddr{martin.ehler@univie.ac.at}}
}

\maketitle
\begin{abstract}
\begin{normalsize}
We study phase retrieval from magnitude measurements of an unknown signal as an algebraic estimation problem. Indeed, phase retrieval from rank-one and more general linear measurements can be treated in an algebraic way. It is verified that a certain number of generic rank-one or generic linear measurements are sufficient to enable signal reconstruction for generic signals, and slightly more generic measurements yield reconstructability for all signals. Our results solve a few open problems stated in the recent literature. Furthermore, we show how the algebraic estimation problem can be solved by a closed-form algebraic estimation technique, termed ideal regression, providing non-asymptotic success guarantees.
\end{normalsize}
\end{abstract}


\section{Introduction}\label{sec:intro}
Intensity measurements in diffraction imaging, microscopy, and $x$-ray crystallography represent magnitudes of Fourier samples, and the recovery of their phases is a difficult problem in optical physics. Within a finite model, phase retrieval is the task of reconstructing a vector in $\mathbb{K}^d$ from the magnitude of finitely many rank-$1$ projections. Classical algorithms are due to Gerchberg/Saxton \cite{Gerchberg:1972kx} and Fienup \cite{Fienup:1982vn} involving alternate projection schemes and fit into standard methods from convex optimization \cite{Bauschke:2002ys}, but signal reconstruction is not guaranteed. Sparse nonconvex optimization is applied in \cite{Balan:2007fk}. Semidefinite programming is used in \cite{Candes:uq}, but success guarantees are only obtained asymptotically with growing dimension. Algebraic reconstruction formulas were derived in \cite{Balan:2009fk}, but require the number of measurements to scale quadratically with the dimension. Jointly, algebraic reconstruction and semidefinite programs were applied in \cite{Bachoc:2012fk} to treat rank-$k$ projectors. For further approaches rooted in signal processing, we refer to \cite{Davidoiu:2012fk,Yagle:1999uq} and references therein.

To successfully reconstruct, measurements must contain sufficient information about the signal. If the number of rank-one magnitude measurements is sufficiently large, then generic measurements allow identifiability of all signals, and there is a range of fewer measurements, in which at least generic signals can still be identified, cf.~\cite{Balan:2006fk}.
Measurements using orthogonal projectors of arbitrary rank have been discussed in \cite{Cahill:2013fk}, from where we cite the following open problems:
\begin{description}
\item[(1)] What is the minimal number of orthogonal projectors enabling phase retrieval for all signals in the real case?
\item[(2)] Do sufficiently many generic orthogonal projectors enable phase retrieval for all signals in the real case?
\item[(3)] Does the minimal number of required orthogonal projectors for retrieving phases for all signals in the complex case depend on the rank of the projectors?
\end{description}
In view of investigating the above mentioned transition range from generic to identifiability of all signals, we derive three additional questions\\

\noindent{\bf(4-6)} by replacing ``for all signals'' in (1-3) with ``generic signals''.\\

Furthermore, the results in~\cite{Balan:2013fk,Bodmann:2013fk} directly lead to one more question, which is formulated as a conjecture in~\cite{Bandeira:2013fk}:
\begin{description}
\item[(7)] Do $4n-4$ generic rank-one measurements allow phase retrieval for all signals in the complex case?
\end{description}

So besides the aim for a better understanding of the structure of phase retrieval in general, we are also left with $7$ open problems that we intend to solve.

In this paper, we claim that phase retrieval is in its core an algebraic problem and emphasize the potential of algebraic tools. This change of perspective enables us to not only answer all of the $7$ above questions, but we can also apply symbolic computations and schemes from approximate algebra to design a reconstruction algorithm. Indeed, we observe that phase retrieval can be tackled by ideal regression as introduced in \cite{Kiraly:2012fk} leading to an algebraic signal reconstruction algorithm for few measurements with nonasymptotic success guarantees.

\subsection*{A short note on question 7}
We would like to note that after submission of this paper, question 7 has independently been answered in~\cite{conca2013} by different techniques. The approach of~\cite{conca2013} is more specifically designed for that question, and uses very explicit computations which are not essentially required to obtain the result since it follows from general principles, as we show. On the same note, the article~\cite{conca2013} contains interesting structural results which can be appreciated independently from question 7. 

\section{The Algebra of Phase Retrieval}\label{sec:theorems}
\subsection{Algebraization of Phase Retrieval}
\label{sec:theorems.algpr}

In this section, we will describe how phase retrieval can be viewed as an algebraic problem. This will be crucial in deriving algebraic solution techniques for phase retrieval. In the usual formulation, the two variants of phase retrieval pose two differently flavoured major obstacles to amenability for algebraic tools: in the real formulation, the mapping is algebraic, but the ground field, the real numbers $\RR$, is not algebraically closed. In the complex formulation, the ground field $\CC$ is algebraically closed, but the measurement mapping includes complex conjugation, making it non-algebraic. The latter problem can be overcome - as it has been demonstrated for example in~\cite{Balan:2006fk}, by treating the real and imaginary part separately, making the mapping algebraic, but the ground field real in its stead, and therefore reducing the second problem to the first one.

We will overcome this obstacle by, again, regarding the algebraic mapping over the complex numbers as base field, and restricting back to the reals when necessary. This procedure will allow us to algebraize the measurement process, derive theoretical bounds on reconstructability, and develop accurate reconstruction algorithms.

First we recapitulate the measurement process:

\begin{Prob}[Phase Retrieval, original version]\label{Prob:original}
Let $\KK=\RR$ or $\KK=\CC$. Let $z\in\mathbb{K}^n$ be an unknown vector. Let $P_1,\dots, P_k\in\mathbb{K}^{r\times n}$ be known matrices. Reconstruct $z$ from the measurements
$$b_i = \|P_i z \|^2 = \Tr (z z^*\cdot P_i^* P_i) ,\quad 1\le i\le k,$$
and the knowledge of the $P_i$.
\end{Prob}

In the usual phase retrieval scenario, the $P_i$ are projectors of rank one. The slightly generalized setting above can be treated with the same mathematical and algorithmical tools, so it means no loss of generality or specifity. Also note that if $\KK=\RR$, then $z$ can be reconstructed only up to sign, and if $\KK=\CC$, then only up to phase.

We will now stepwise reformulate the problem, in order to make it amenable to algebraic tools. First we note that phase retrieval is known to be an inverse problem. That is, there is a so-called forward mapping, which takes the (unknown to the observer) signal $z$, and outputs the (observed) values $b_i$. The backward problem is then to obtain $z$ from the $b_i$. Since $z$ can be obtained only up to sign or phase, this is equivalent to obtaining the matrix $Z=zz^*$. Writing all of this explicitly, we obtain as a reformulation of the original Problem~\ref{Prob:original} the following inverse problem:

\begin{Prob}\label{Prob:inverse}
Let $\KK=\RR$ or $\KK=\CC$. Consider the forward mapping
\begin{align*}
\phi:& \left(\KK^{r\times n}\right)^k\times \KK^{n\times n}\rightarrow \left(\KK^{r\times n}\right)^k \times \KK^k\\
& (P_1,\dots, P_k, Z)\mapsto \left(P_1,\dots, P_k, \Tr( Z\cdot P_1^* P_1),\dots, \Tr (Z\cdot P_k^* P_k)\right).
\end{align*}
Reconstruct $\tau:=(P_1,\dots, P_k, Z)$, given $\phi(\tau)$, and assuming that $Z$ is rank one and Hermitian.
\end{Prob}

Note that we have deliberately included the $P_i$ in the range and the image of $\phi$, in order to mathematically model the fact that the projectors $P_i$ are known to the observer; and for technical reasons - equivalent to the latter - which will become apparent further on. Furthermore, assuming that $Z$ is rank one and Hermitian is equivalent to assuming that $Z=zz^*$ for suitable $z$, since knowing $Z$ is equivalent to know $z$ up to sign/phase.

As said in the beginning, there are two major difficulties in applying algebraic techniques to Problem~\ref{Prob:inverse}. The first is that (A) the base field is not algebraically closed if $\KK=\RR$, the second being that (B) the mapping $\phi$ is not algebraic if $\KK=\CC$, since it includes complex conjugation. The solution approach for problem (A) is relatively straithgforward: since the mapping $\phi$ includes only transposes, it is algebraic, therefore we consider the same mapping over the complex numbers. Also, we replace the matrices $P_i\in\RR^{r\times n}$ by matrices $A_i:=P_i^\top P_i$ for reason of convenience:

\begin{Prob}\label{Prob:complex}
Let $z\in\CC^n$ be an unknown vector. Consider the forward mapping
\begin{align*}
\phi:& \left(\CC^{n\times n}\right)^k\times \CC^{n\times n}\rightarrow \left(\CC^{n\times n}\right)^k \times \CC^k\\
& (A_1,\dots, A_k, Z)\mapsto \left(A_1,\dots, A_k, \Tr( Z\cdot A_1),\dots, \Tr (Z\cdot A_k)\right)
\end{align*}
Reconstruct $\tau:=(A_1,\dots, A_k, Z)$, given $\phi(\tau)$, and assuming that $Z$ is symmetric rank one, and that the $A_i$ are symmetric of rank $r$.
\end{Prob}

There are now several things to note: first, the map $\phi$ is algebraic, and range and image are now complex. In particular, the measurements can be complex. Note that we want both $Z$ and $A_i$ to be symmetric, not Hermitian, otherwise the problem would not be algebraic.

Most importantly, however, Problem~\ref{Prob:complex} is a problem which is a-priori different from Problem~\ref{Prob:inverse}, since we have enlarged image and range. When restricting to reals, we obtain the original phase retrieval Problem~\ref{Prob:inverse}, but there is no a-priori reason to believe that the behavior of the complex variant is fundamentally the same as for the original problem.

However, as will turn out, Problem~\ref{Prob:complex} is much easier amenable to tools from algebraic geometry, both on the theoretical and the practical side, and results and algorithms will give rise to solutions for questions and tasks over the reals, as it will be explained in the following section.\\

We proceed treating the variant of the phase retrieval problem~\ref{Prob:inverse} where complex signals are allowed. Recall that the problem was that (B) the map $\phi$ is not algebraic. The solution for this is to ``algebraize'' the map by considering real and imaginary part separately. Namely, writing $P_i = Q_i + \iota\cdot S_i$ with $Q_i,S_i\in\mathbb{R}^{m\times n}$ and $z=x+\iota y$, where $\iota$ denotes the imaginary unit, we obtain:

\begin{Prob}\label{Prob:inversecompalgreal}
Let $x,y \in\mathbb{R}^{n}$ be unknown vectors, write $R:=xx^\top + yy^\top$ and $\Phi:=yx^\top -xy^\top$. Also, write $B_i:=Q_i^\top Q_i + S_i^\top S_i$ and $C_i:= Q_i^\top S_i - S_i^\top Q_i$ for $Q_i,S_i\in\mathbb{R}^{m\times n}$. Consider the forward mapping
\begin{align*}
\phi:& \left(\RR^{n\times n}\right)^{2k}\times \RR^n \rightarrow \left(\RR^{n\times n}\right)^{2k} \times \RR^k\\
& (B_1,C_1,\dots, B_k,C_k, R,\Phi)\mapsto \left(B_1,C_1,\dots, B_k,C_k, \Tr(R\cdot B_1+ \Phi\cdot C_1) ,\dots, \Tr(R\cdot B_n+ \Phi\cdot C_n)\right)
\end{align*}
Reconstruct $\tau= (B_1,C_1,\dots, B_k,C_k,R,\Phi)$, given $\phi(\tau)$, assuming that $B_i,C_i,R,\Phi$ were of the above form.
\end{Prob}

An elementary computation shows that Problem~\ref{Prob:inversecompalgreal} is equivalent to the original complex phase retrieval problem~\ref{Prob:original}: namely, $zz^*= R +\iota\Phi$, so knowing $R$ and $\Phi$ is equivalent to knowing $z$ up to phase. Observe that $\phi$ is now an algebraic map, since the rule is algebraic, and so is the possible set of $B_i,C_i,X,Y$. However, the mapping $\phi$ is now over the reals, a field which is not algebraically closed, entailing an analogue of complication (A) which we have treated in the real case by allowing complex matrices in the range. We will once more do the same and allow a complex range. The set of matrices though have a very specific structure, so we introduce notation for them in our final formulation of the complex phase retrieval problem:

\begin{Prob}[algebraized phase retrieval of complex signal]\label{Prob:inversecompalg}
Define the following sets of matrices:
\begin{align*}
\calS_\CC&:=\{(xx^\top + yy^\top, yx^\top - xy^\top)\;:\;x,y\in\CC^n\}\subseteq \CC^{n\times n}\times \CC^{n\times n}\\
\calP_\CC(r)&:=\{(Q^\top Q + S^\top S,Q^\top S - S^\top Q) \;:\;S,Q\in\CC^{r\times n}\}\subseteq \CC^{n\times n}\times \CC^{n\times n}
\end{align*}
Consider the forward mapping
\begin{align*}
\phi:& \calP_\CC(r)^{k}\times \calS_\CC \rightarrow \calP_\CC(r)^{k} \times \CC^k\\
& (B_1,C_1,\dots, B_k,C_k, R,\Phi)\mapsto \left(B_1,C_1,\dots, B_k,C_k, \Tr(R\cdot B_1+ \Phi\cdot C_1) ,\dots, \Tr(R\cdot B_n+ \Phi\cdot C_n)\right)
\end{align*}
Given $\tau=\phi(B_1,C_1,\dots, B_k,C_k,R,\Phi)$, determine $\phi^{-1}(\tau)$.
\end{Prob}

The set $\calS_\CC$ parameterizes the possible signals, while $\calP_\CC(r)$ parameterizes the possible projections (of rank $r$). Note that $\calS_\CC=\calP_\CC(1)$; nevertheless we make this notational distinction between $\calS_\CC$ and $\calP_\CC(.)$ for clarity.

We reformulate the phase retrieval problem for real signals in analogy, by defining symbols for the space of matrices, yielding in the final version:

\begin{Prob}[algebraized phase retrieval of real signal]\label{Prob:inverserealalg}
Define the following sets of matrices:
\begin{align*}
\calS_\rho &:=\{zz^\top \;:\;z\in\CC^n\}\subseteq \CC^{n\times n}\\
\calP_\rho(r)&:=\{P_i^\top P_i \;:\;P_i\in\CC^{r\times n}\}\subseteq \CC^{n\times n}
\end{align*}
Consider the forward mapping
\begin{align*}
\phi:& \calP_R(r)^{k}\times \calS_R \rightarrow \calP_R (r)^{k} \times \CC^k\\
& (A_1,\dots, A_k, Z)\mapsto \left(A_1,\dots, A_k, \Tr(Z\cdot A_1) ,\dots, \Tr(Z\cdot A_n)\right)
\end{align*}
Given $\tau=\phi(A_1,\dots, A_k,Z)$, determine $\phi^{-1}(\tau)$.
\end{Prob}
Observe that $\calS_\rho$ models the possible signals, and is exactly the set of symmetric complex matrices of rank $1$ (or less), whereas $\calP_\rho(r)$ models the projections, and is exactly the set of symmetric complex matrices of rank $r$ (or less). Note that we have formulated both the real and the complex problem with almost the same forward mapping, the difference lies in the different sets of projection matrices, where in the real case we have single matrices, and in the complex case we have related pairs. Also, for the complex variant of phase retrieval, we have related pairs of matrices $R$ and $\Phi$ instead of the single matrix $Z$.\\

In order to make the notation uniform for both the real and complex cases, we introduce the following convention:
\begin{Not}
Let $Z,A\in\CC^{n\times n}\times \CC^{n\times n}$, with $Z=(X,Y)$ and $A=(B,C)$. Then, we will write, by convention,
$$\Tr (Z\cdot A):=\Tr (X\cdot B + Y\cdot C).$$
\end{Not}

\subsection{Identifiability and Genericity}

A signal $z$ is called \emph{identifiable} if it is uniquely determined in $\KK^n$ by the measurements $b_i$ up to a global phase factor, which is an ambiguity one cannot avoid. The choice of $k$ generic measurements by means of rank-$1$ projectors yield identifiability of generic signals if and only if $k\geq n+1$ in the real and $k\geq 2n$ in the complex case, cf.~\cite[Theorems~2.9 and~3.4]{Balan:2006fk}. Generic rank-$1$ projectors yield identifiability for all signals if and only if $k\geq 2n-1$ in the real case. For the complex setting, examples with $k\geq 4n-4$ that allow reconstruction are known, and this bound is conjectured to be necessary \cite{Bodmann:2013fk}.

We will generalize the statements to the scenario of general linear projections. As described earlier, the strategy is to consider first the corresponding algebraized problem over an algebraically closed field, namely $\CC$, instead of $\RR$, and then descend the results back to the real numbers $\RR$. Again, it is important to note that this is subtly different from considering the projection problem over the complex numbers, since instead of complex conjugation, we consider transposition in order to keep the problem algebraic.

\subsubsection*{A Short Note on Technical Conditions}
The following exposition will use some technical conditions on varieties and maps, namely them being \emph{irreducible}, and (generically) \emph{unramified}. These are standard notions in algebraic geometry and can be found in most introductory books - we refrain from explaining them here as this is beyond the scope of the paper; the logic in the proofs can be understood without knowing what these mean exactly - a glossary of definitions can be found in Appendix~\ref{app:algebraic-glossary}. Intuitively, an algebraic set being irreducible means that there is only one prototypical behaviour for its elements. Unramifiedness is a point-wise algebraic certificate for a mapping staying stable under perturbation in a certain sense. In our case, unramifiedness will certify for identifiability which is stable under perturbation of signals or measurements.

\subsubsection{Identifiability of Signals}\label{sec:theorems.signals}
In this paragraph, we translate identifiability of a signal into an algebraic statement. The main concepts will be identifiability, and identifiability which is stable under perturbation, both corresponding to certain algebraic properties of the signal.

\begin{Not}
We fix some notation and technical assumptions that will be valid in the relevant cases of real and complex phase recognition:
\begin{description}
\item[(i)] The signals will be modelled by an irreducible variety $\calS\subseteq \left(\CC^{n\times n}\right)^\gamma,$ with $\gamma=1$ in the real and $\gamma=2$ in the complex case. For example, $\calS=\calS_\rho$ or $\calS=\calS_\CC$, as in Section~\ref{sec:theorems.algpr}.
\item[(ii)] A measurement scheme will be modelled by the tuple $A=(A_1,\dots, A_k)\in \left(\left(\CC^{n\times n}\right)^\gamma\right)^k$ with $k\in\NN$ being the number of measurements.
\item[(iii)] The measurement process is the formal forward mapping
$$ \phi_A: \calS \rightarrow \CC^k,\quad Z\mapsto \left(\Tr(Z\cdot A_1) ,\dots, \Tr(Z\cdot A_k)\right).$$
\end{description}
\end{Not}

The condition that $\calS$ is irreducible is fulfilled in the cases discussed in the introductory Section~\ref{sec:theorems.algpr}. Namely, both $\calS_\rho$ and $\calS_\CC$ are irreducible varieties, as it is proved in Proposition~\ref{Prop:irrvarieties}

We recapitulate a statement from the last section which expresses identifiability in this formal, slightly more technical setting:

\begin{Rem}\label{Rem:identalg}
By definition, the following are the same:
\begin{description}
   \item[(i)] $Z\in\calS$ is identifiable from $\phi_A(Z)$.
   \item[(ii)] $\# \phi_A^{-1}\phi_A(Z) = 1.$
\end{description}
\end{Rem}

The following statement is crucial in obtaining our local-to-global principle for identifiability. It characterizes signals which are identifiable and stably so under perturbation just in terms of the signal itself, therefore allowing to remove any reference to open neighbourhoods.

\begin{Prop}\label{Prop:algcritsignal}
Assume that $\phi_A$ is generically unramified. Let $Z\in\calS$. Then, the following three statements are equivalent:
\begin{description}
   \item[(i)] $Z$ is identifiable from $\phi_A(Z)$, and remains identifiable under infinitesimal perturbation. (That is, there is a relatively Borel-open neighborhood $U\subseteq \calS$ with $Z\in U$ such that for all $Y\in U$, it holds that $\card{ \phi_A^{-1}\phi_A(Z)} = 1.$)
   \item[(ii)] $Z$ is identifiable from $\phi_A(Z)$, and $\phi_A$ is unramified over $Z$.
   \item[(iii)] A generic $Y\in\calS$ is identifiable from $\phi_A(Z)$. (That is, the set of non-identifiable $Y\in\calS$ is a proper Zarkiski closed subset and therefore Hausdorff measure zero subset of $\calS$.)
\end{description}
In particular, condition (i) is a Zariski open condition on the signal $Z$; that is, the set of signals $Z$ with property (i) is a Zariski open subset of $\calS$.
\end{Prop}
\begin{proof}
This is implied by Proposition~\ref{Prop:opencert} in the appendix.
\end{proof}

The condition that $\phi_A$ is generically unramified is slightly technical and fulfilled in the prototypical cases, see Proposition~\ref{Prop:unrphi} in the appendix for a proof. The condition that $\phi_A$ is \emph{unramified} over $Z$, on the other hand, is the crucial local property to which we translate perturbation-stability. Intuitively, Proposition~\ref{Prop:algcritsignal}~(iii) means that an identifiable signal which remains so under perturbation certifies for the whole signal space. It is also important to note that condition~(ii) in Proposition~\ref{Prop:algcritsignal} is essentially independent from the choice of $\calS$ while (i) and (iii) are a-priori not. We introduce terminology for the condition described in~(i):

\begin{Def}
For brevity, we will call a signal $Z\in\calS$ that is identifiable from $\phi_A(Z)$, and remains identifiable under infinitesimal perturbation, a \emph{perturbation-stably identifiable} signal (by Proposition~\ref{Prop:algcritsignal}~(ii), this is equivalent to $Z$ being identifiable and $\phi_A$ being unramified over $Z$).
\end{Def}

We can reformulate Proposition~\ref{Prop:algcritsignal} as a principle of excluded middle, stating that either almost all signals are perturbation-stably identifiable, or none:

\begin{Cor}\label{Cor:genericpsid}
\begin{description}
   \item[(i)] If there exists a signal $Z\in\calS$ which is perturbation-stably identifiable from $\phi_A (Z)$, then a random signal $Y\in\calS$ is perturbation-stably identifiable with probability one under any Hausdorff continuous probability density on $\calS$.
   \item[(ii)] It cannot happen that there are sets $\calA,\calB\subseteq \calS$, both with positive Hausdorff measure, such that all signals $Z\in\calA$ are perturbation-stably identifiable, and all signals $Z\in\calA$ are not perturbation-stably identifiable.
\end{description}
\end{Cor}
\begin{proof}
This is a direct consequence of Proposition~\ref{Prop:algcritsignal}, using that by taking Radon-Nikodym derivatives, $\calS$-Hausdorff zero sets are as well probability measure zero sets for any continuous probability measure.
\end{proof}

We give an example to illustrate that for a signal, being identifiable is different from being perturbation-stably identifiable:

\begin{Ex}\label{Ex:psidvsid}
Consider the situation where the signals $\calS =\{zz^\top\;:\; z\in\CC^3\}$ are symmetric rank one matrices (i.e., the usual phase recognition setting), and our matrices are chosen as
$$ A_1 =   \left(\begin{array}{ccc}
      1 & 0 & 0\\
      0 & 0 & 0\\
      0 & 0 & 0
      \end{array}\right),\quad
A_2 =   \left(\begin{array}{ccc}
      0 & 0 & 0\\
      0 & 1 & 0\\
      0 & 0 & 0
      \end{array}\right),\quad
A_3 =   \left(\begin{array}{ccc}
      0 & 0 & 0\\
      0 & 0 & 0\\
      0 & 0 & 1
      \end{array}\right), \quad
A_4 =   \left(\begin{array}{ccc}
      1 & 1 & 1\\
      1 & 1 & 1\\
      1 & 1 & 1
      \end{array}\right).$$
For $\ell=1,2,3$, write $\calZ_\ell:=\{zz^\top \;:\;z\in\CC^3, z_i=-z_j,\;\mbox{where}\; \{1,2,3\}=\{\ell,i,j\}\}$ and $\calC_\ell:=\{zz^\top \;:\;z\in\CC^3, z_\ell = 0\}$. Write $\calZ=\calZ_1\cup\calZ_2\cup\calZ_3$ and $\calC=\calC_1\cup\calC_2\cup\calC_3$. Then, one can check, by an elementary computation:
\begin{description}
   \item[(i)] the perturbation-stably identifiable signals are exactly the signals in $\CC^{3\times 3}\setminus \calZ$.
   \item[(ii)] the signals in $\calZ\setminus \calC$ are not identifiable.
   \item[(iii)] the signals in $\calZ\cap \calC$ are identifiable, but due to (ii) not perturbation-stably identifiable.
\end{description}
Note that $\calZ\cap \calC$ is the image of six lines under the map $z\mapsto zz^\top$, three of which are the coordinate axes.
\end{Ex}

\subsubsection{Identifyingness as a Measurement Property}\label{sec:theorems.measurem}

In Corollary~\ref{Cor:genericpsid}, it has been shown that if one signal is perturbation-stably identifiable, then almost all signals are. Therefore the fact whether almost all signals are identifiable can be regarded as a property of the measurement regime. The following theorem makes this statement exact and states that measurement regimes fall into exactly one of three classes:

\begin{Thm}\label{Thm:algsignalshort}
For a fixed measurement regime $(A_1,\dots, A_k),$ consider the three cases
\begin{description}
   \item[(a)] A generic signal $Z\in\calS$ is not identifiable from $\phi_A(Z)$.
   \item[(b)] A generic, but not all signals $Z\in\calS$, are identifiable from $\phi_A(Z)$.
   \item[(c)] All signals $Z\in\calS$ are identifiable from $\phi_A(Z)$.
\end{description}
The three cases above are mutually exclusive and exhaustive, and equivalent to
\begin{description}
   \item[(a)] No signal $Z\in\calS$ is perturbation-stably identifiable from $\phi_A(Z)$.
   \item[(b)] A generic, but not all signals $Z\in\calS$, are perturbation-stably identifiable from $\phi_A(Z)$.
   \item[(c)] All signals $Z\in\calS$ are perturbation-stably identifiable from $\phi_A(Z)$.
\end{description}
\end{Thm}
\begin{proof}
This is implied by Theorem~\ref{Thm:algsignal} in the appendix.
\end{proof}

Recall Example~\ref{Ex:psidvsid} which shows that the set of identifiable and perturbation-stably identifiable signals in case (b) of Theorem~\ref{Thm:algsignalshort} may differ. The following example shows that the set of identifiable and perturbation-stably identifiable signals may differ in case (a) of Theorem~\ref{Thm:algsignalshort} as well.

\begin{Ex}\label{Ex:psidvsidlow}
Consider the situation where the signals $\calS =\{zz^\top\;:\; z\in\CC^3\}$ are symmetric rank one matrices (i.e., the usual phase recognition setting), and our matrices are chosen as
$$ A_1 =   \left(\begin{array}{ccc}
      1 & 0 & 0\\
      0 & 0 & 0\\
      0 & 0 & 0
      \end{array}\right),\quad
A_2 =   \left(\begin{array}{ccc}
      0 & 0 & 0\\
      0 & 1 & 0\\
      0 & 0 & 0
      \end{array}\right),\quad
A_3 =   \left(\begin{array}{ccc}
      0 & 0 & 0\\
      0 & 0 & 0\\
      0 & 0 & 1
      \end{array}\right).$$
For $\ell=1,2,3$, write $\calC_\ell:=\{zz^\top\;:\; z\in\CC^3, z_\ell = 0\}$ and $\calC=\calC_1\cup\calC_2\cup\calC_3$.
One can check, by an elementary computation:
\begin{description}
   \item[(i)] No signal is perturbation-stably identifiable.
   \item[(ii)] the signals in $\CC^{3\times 3}\setminus\calC$ are not identifiable.
   \item[(iii)] the signal in $\calC$ are identifiable, but not perturbation-stably identifiable.
\end{description}
This is the not the simplest example of this kind, since choosing $n=1$ and one single non-zero $(1\times 1)$-matrix exposes the same behavior, with the origin being the simple identifiable point. However, it is informative to compare this example to Example~\ref{Ex:psidvsid}, in which one more measurement is taken. In this example, the ramification locus (= the set of ramified points) is the union of the coordinate axes $\calC$, a Zariski closed set. There is no non-empty set such that the restriction of $\phi_A$ to it unramified or bijective. In particular, while $\phi_A$ restricted to the origin $(0,0,0)$ is bijective as a map, and therefore an isomorphism of sets, it is generically ramified, since the origin is contained in $\calC$.

Moreover, the identifiable signals in Example~\ref{Ex:psidvsid} consist exactly of the union of this Zariski closed set, and the Zariski open set of perturbation-stably identifiable signals, explaining why the set of identifiable signals in the other example are neither Zariski closed nor Zariski open.
\end{Ex}

Theorem~\ref{Thm:algsignalshort} allows to regard the different grades of identifiability (a), (b), (c) as properties of the measurement regime. We therefore introduce the following abbreviating notation:

\begin{Def}\label{def:measures stable def}
We call a measurement tuple $A=(A_1,\dots, A_k)$:
\begin{description}
   \item[(a)] \emph{non-identifying} for signals in $\calS$, if no signal $Z\in\calS$ is perturbation-stably identifiable from $\phi_A(Z)$.
   \item[(b)] \emph{generically identifying} for signals in $\calS$, if generic signals $Z\in\calS$ are (perturbation-stably) identifiable from $\phi_A(Z)$, and \emph{incompletely identifying}, if generic, but not all signals $Z\in\calS$ are (perturbation-stably) identifiable from $\phi_A(Z)$.
   \item[(c)] \emph{completely identifying} for signals in $\calS$, if all signals $Z\in\calS$ are (perturbation-stably) identifiable from $\phi_A(Z)$.
\end{description}
If $\calS$ is obvious from the context, we will omit the qualifier ``for signals in $\calS$'', always keeping in mind that the terminology depends on $\calS$.
\end{Def}

Theorem~\ref{Thm:algsignalshort} then can be rephrased that a measurement regime $A_1,\dots, A_k$ is either non-identifying, incompletely identifying, or completely identifying - note that due to the theorem, it does not matter whether the ``perturbation-stably'' in the brackets is there or not. We will now show that being non-identifying, generically and completely identifying are properties of the space of possible measurements, just as identifiability is not only a property of the signal, but of signal space.

\begin{Not}
We introduce some notation modelling the space of measurements:
\begin{description}
\item[(iv)] The space of measurements of type $(A_1,\dots, A_k)$ will be modelled by irreducible varieties $\calP_1,\dots, \calP_k \subseteq \left(\CC^{n\times n}\right)^\gamma,$ with $\gamma=1$ in the real and $\gamma=2$ in the complex case. We will write $\calP^{(k)}=\calP_1\times\dots\times \calP_k$ for the space of measurement tuples of size $k$. For example, $\calP^{(k)}=\calP_\CC (r)^k$ for complex signals, or $\calP^{(k)}=\calP_\rho(r)^k$ for real ones.
\item[(v)] The extended measurement process will be modelled by the formal forward mapping
\begin{align*}
\phi:& \calP^{k}\times \calS \rightarrow \calP^{k} \times \CC^k\\
& (A_1,\dots, A_k, Z)\mapsto \left(A_1,\dots, A_k, \Tr(Z\cdot A_1) ,\dots, \Tr(Z\cdot A_n)\right)
\end{align*}
\end{description}
\end{Not}

The condition that the $\calP_i$ is irreducible is fulfilled in the cases discussed in the introductory Section~\ref{sec:theorems.algpr}: both $\calP_\rho(r)$ and $\calP_\CC (r)$ are irreducible varieties, see Proposition~\ref{Prop:irrvarieties}. Our main result is an analogue to the characterization in Proposition~\ref{Prop:algcritsignal}, now for the measurement matrices:

\begin{Prop}\label{Prop:algcritmatrices}
Assume that $\phi$ is generically unramified. Then, the following three statements are equivalent:
\begin{description}
   \item[(i)] $Z\in\calS$ is identifiable from $\phi(A,Z)$, and remains identifiable under infinitesimal perturbation of $A$ and $Z$. (That is, there is a relatively Borel-open neighborhood $U\subseteq \calP^{k}\times\calS$ with $(A,Z)\in U$ such that for all $Y\in U$, it holds that $\card{ \phi^{-1}\phi(Y)} = 1.$)
   \item[(ii)] $(A,Z)$ is identifiable from $\phi (A,Z)$, and $\phi$ is unramified over $(A,Z)$.
   \item[(iii)] For generic $B\in\calP^{k}$, a generic $Y\in\calS$ is identifiable from $\phi(A,Y)$. (That is, the set of $(B,Y)\in \calP^{k}\times\calS$ where $Y\in\calS$ is non-identifiable from $\phi (B,Y)$ is a proper Zarkiski closed subset and therefore Hausdorff measure zero subset of $\calP^{k}\times \calS$.)
\end{description}
In particular, condition (i) is a Zariski open property on the measurement-signal-pair $(A,Z)$; that is, the set of measurement-signal-pairs $(A,Z)$ with property (i) is a Zariski open subset of $\calP^{k}\times \calS$.
\end{Prop}
\begin{proof}
The statement follows from Proposition~\ref{Prop:opencert}, applied to the irreducible variety $\calX=\calP^{k}\times \calS$.
\end{proof}

The main obstacle in generalizing Theorem~\ref{Thm:algsignalshort} to an algebraic characterization, or a local-global-property of measurements lies in the fact that the perturbation can occur in both the signal $Z$ and the measurement regime $A$. We therefore need to provide an intermediate result which removes the dependence on the measurement:

\begin{Prop}\label{Prop:opencond}
Assume that $\phi$ is generically unramified. Then, the following two conditions on measurement regimes $A\in \calP^{k}$ are (Zariski) open conditions:
\begin{description}
   \item[(i)] $A$ is generically identifying and remains generically identifying under perturbation. That is, there is a (relatively Borel-) open neighborhood $U\subseteq \calP^{k}$ with $A\in U$ such that all $B\in U$ are generically identifying.
   \item[(ii)] $A$ is completely identifying and remains completely identifying under perturbation. That is, there is a (relatively Borel-) open neighborhood $U\subseteq \calP^{k}$ with $A\in U$ such that all $B\in U$ are completely identifying.
\end{description}
\end{Prop}
\begin{proof}
Consider the maps
\begin{align*}
\phi: &\calP^{(k)}\times \calS \rightarrow \calP^{(k)} \times \CC^k,\; (A_1,\dots, A_k, Z)\mapsto \left(A_1,\dots, A_k, \Tr(Z\cdot A_1),\dots,\Tr(Z\cdot A_n)\right),\\
\psi: &\calP^{(k)}\times \calS \rightarrow \calP^{(k)},\; (A_1,\dots, A_k, Z)\mapsto (A_1,\dots, A_k),\\
\pi: &\calP^{(k)}\times \calS \rightarrow \calS, \;(A_1,\dots, A_k, Z)\mapsto Z.
\end{align*}
(i) Consider the set $\calY=\{x\in\calP^{(k)}\times \calS\;:\;\phi(x)\;\mbox{is identifiable and unramified}\}.$ By Proposition~\ref{Prop:algcritmatrices} $\calZ$ is a Zariski open set (and possibly empty). Since $\psi$ is surjective, the set $\psi(\calY)$ is therefore an open subset of $\calP^{(k)}$, and by construction, describes the condition (i), therefore proving its openness.\\

(ii) Keep the notations above, and consider the set-complement $\calY^C$ of $\calY$ in $\calP^{(k)}\times \calS$. Since $\calY$ is open, $\calY^C$ is closed, and $\calV:=\psi\left(\calY^C\right)$ is closed as well. Therefore, the set-complement $\calV^C$ in $\calP^{(k)}$ is open. By construction, $\calV^C$ describes condition (ii), therefore openness of condition (ii) follows.
\end{proof}

\begin{Def}
We call a measurement regime $A\in\calP^{(k)}$:
\begin{description}
   \item[(a)] \emph{stably non-identifying} in $\calP^{(k)}$, if $A$ is non-identifying and remains non-identifying under perturbation, as in Proposition~\ref{Prop:opencond}~(i).
   \item[(b)] \emph{stably generically identifying} in $\calP^{(k)}$, if $A$ is generically identifying and remains generically identifying under perturbation, as in condition (i). \emph{stably incompletely identifying} in $\calP^{(k)}$, if $A$ is incompletely identifying and remains incompletely identifying under perturbation, as in Proposition~\ref{Prop:opencond}~(i).
   \item[(c)] \emph{stably completely identifying} in $\calP^{(k)}$, if $A$ is completely identifying and remains completely identifying under perturbation, as in Proposition~\ref{Prop:opencond}~(ii).
\end{description}
If $\calP^{(k)}$ is obvious from the context, we will omit the qualifier ``in $\calP^{(k)}$'', always keeping in mind that the terminology depends on $\calP^{(k)}$.
\end{Def}

Note that a measurement regime can be, at the same time, neither stably non-identifying, stably generically identifying, nor stably completely identifying. However, by definition, stably non-identifying is mutually exclusive to stably generically identifying, and stably incompletely identifying is mutually exclusive to stably completely identifying.

Proposition~\ref{Prop:opencond} can be reformulated as a local-to-global-principle:

\begin{Cor}\label{Cor:genericmeas}
Keep the notations of Proposition~\ref{Prop:opencond}. Then:
\begin{description}
   \item[(ia)] If there exists a stably generically identifying measurement regime $A\in\calP^{(k)}$, a generic measurement $B\in\calP^{(k)}$ is stably generically identifying.
   \item[(ib)] If there exists a stably completely identifying measurement regime $A\in\calP^{(k)}$, a generic measurement $B\in\calP^{(k)}$ is stably completely identifying.
   \item[(iia)] If there exists a stably non-identifying measurement regime $A\in\calP^{(k)}$, a generic measurement $B\in\calP^{(k)}$ is stably non-identifying, and no measurement is stably incompletely identifying or stably completely identifying.
   \item[(iib)] If there exists a stably incompletely identifying measurement regime $A\in\calP^{(k)}$, a generic measurement $B\in\calP^{(k)}$ is stably incompletely identifying, and no measurement is stably completely identifying.
\end{description}
\end{Cor}
\begin{proof}
(ia) and (ib) follow from Zariski-openness of the conditions asserted in Proposition~\ref{Prop:opencond}. (iia) and (iib) can be derived as negations.
\end{proof}

Proposition~\ref{Prop:opencond} also allows to prove an analogue of Theorem~\ref{Thm:algsignalshort}, now for classes of measurements instead of a single measurement regime:

\begin{Thm}\label{Thm:algmeasureshort}
Assume that $\phi$ is generically unramified. Consider the three cases
\begin{description}
   \item[(a)] A generic measurement regime $A\in\calP^{k}$ is non-identifying.
   \item[(b)] A generic measurement regime $A\in\calP^{k}$ is incompletely identifying.
   \item[(c)] A generic measurement regime $A\in\calP^{k}$ is completely identifying.
\end{description}
The three cases above are mutually exclusive and exhaustive, and equivalent to
\begin{description}
   \item[(a)] A generic measurement regime $A\in\calP^{k}$ is stably non-identifying. No measurement regime $A\in\calP^{k}$ is stably generically identifying.
   \item[(b)] A generic measurement regime $A\in\calP^{k}$ is stably incompletely identifying.
   \item[(c)] A generic measurement regime $A\in\calP^{k}$ is stably completely identifying.
\end{description}
\end{Thm}
\begin{proof}
A proof by analogy is described in the appendix, as Theorem~\ref{Thm:algmeasure}
\end{proof}

We can therefore define terminology that describe cases (a) to (c) shortly:

\begin{Def}
Keep the notations of Theorem~\ref{Thm:algmeasureshort}. We will call a the set of measurements $\calP^{k}$ \emph{generically unramified} if $\phi$ is generically unramified. We will call a generically unramified $\calP^{k}$:
\begin{description}
   \item[(a)] \emph{non-identifying} if a generic measurement $A\in\calP^{k}$ is non-identifying.
   \item[(b)] \emph{generically identifying} if a generic measurement $A\in\calP^{k}$ is generically identifying. \emph{incompletely identifying} if a generic measurement $A\in\calP^{k}$ is incompletely identifying.
   \item[(c)] \emph{completely identifying} if a generic measurement $A\in\calP^{k}$ is completely identifying.
\end{description}
\end{Def}

Note that, somewhat differently as it happens in Theorem~\ref{Thm:algsignalshort} for signals, generically identifying measurements can exist in generically non-identifying measurement regimes - with other words, there can be a measurement in $\calP^{k}$ which identifies a generic signal, while a generic measurement in $\calP^{k}$ does not identify a generic signal. This does not contradict the above discussion since a generically non-identifying measurement regime does include perturbation-stability with respect to the measurements, while a generically idenfitying measurement does not. Algebraically spoken, $\phi_A$ does not ramify for generic signals $Z$, while $\phi$ is ramified at $(A,Z)$ for all $Z$. An explicit example for this behaviour is given in section~\ref{sec:theorems.ramex}.

\subsection{Transfer Results for Identifyingness}
In this section we will collect different results that allow to transfer identifiyingness properties from one set of potential measurements to another:

\begin{Not}
We will consider irreducible varieties $\calP^{(k)}=\calP_1\times \dots \times \calP_k$ and $\calQ^{(k)}=\calQ_1\times\dots\times \calQ_k$, with corresponding forward maps $\phi,\varphi$.
\end{Not}

The first lemma allows to obtain identifiability for a broader measurement space, if it is already established for a smaller:

\begin{Lem}\label{Lem:widen}
Assume $\calP_i\subseteq \calQ_i$ for all $i$, that is, $\calP^{(k)}\subseteq \calQ^{(k)}$. Then:
\begin{description}
   \item[(i)] If $\calP^{(k)}$ is generically unramified, then so is $\calQ^{(k)}$.
   \item[(ii)] If $\calP^{(k)}$ is generically identifying, then so is $\calQ^{(k)}$.
   \item[(iii)] If $\calP^{(k)}$ is completely identifying, then so is $\calQ^{(k)}$.
\end{description}
\end{Lem}
\begin{proof}
(i) follows from Theorem~\ref{Thm:openconds}~(ii) in the appendix. For (ii), the characterization in Theorem~\ref{Thm:algmeasure} yields that is birational. Therefore, there is $(A,Z)\in \calP^{(k)}\times \calS$ above which $\phi$ is unramified and for which $\card{\phi^{-1}\phi(A,Z)}=1$. Since $\phi$ remembers $A$ exactly, this is equivalent to $\card{\varphi^{-1}\varphi(A,Z)}=1$; also $\varphi$ is unramified above $(A,Z)$. We can therefore apply Proposition~\ref{Prop:opencert} to infer that $\varphi$ is birational, which implies the statment by Theorem~\ref{Thm:algmeasure}. (iii) follows in analogy, repeating the argument for all $Z\in\calS$.
\end{proof}
As Example~\ref{Ex:psidvsidlow} shows, the reverse implications do not hold in general.

We prove another lemma, which is specific to the case of matrices stratified by rank, and which will allow to restrict to orthogonal or unitary matrices, once properties of all matrices of fixed rank are established:

\begin{Lem}\label{Lem:orthdescent}
Assume the $\calP_i\subseteq \left(\CC^{n\times n}\right)^\gamma$ are all spaces of rank at most $r_i$ matrices, that is, of the form $\calP_\rho(r_i)$ or $\calP_\CC(r_i)$. Assume that the $\calQ_i$ are the corresponding variety of orthogonal/unitary projection matrices of rank exactly $r_i$. Then:
\begin{description}
   \item[(i)] $\calP^{(k)}$ is generically identifying if and only if $\calQ^{(k)}$ is.
   \item[(ii)] $\calP^{(k)}$ is completely identifying if and only if $\calQ^{(k)}$ is.
\end{description}
\end{Lem}
\begin{proof}
Let $A_1,\dots, A_k$ be generic in $\calP_1,\dots, \calP_k$; we will treat the $A_i$ as single matrices. Since $z^* A z = \frac{1}{2} z^* \left(A^* + A\right) z$, we can assume that the $A_i$ are symmetric/Hermitian and generic. (i) $(A_1,\dots, A_k)$ are generically identifying if for generic $z$, one can reconstruct $z$ up to phase/sign from the $z^* A_i z$. By definition, there is an invertible matrix $S_1$ such that $U_1 = S_1^* A_1 S_1$ is an orthogonal/unitary projector of rank $a_i$. Since $S_1$ is invertible, a vector $z$ is generic if and only if the vector $S_1\cdot z$ is generic, therefore $(A_1,\dots, A_k)$ is generically identifying if and only if $(U_1, S_1^* A_2 S_1,\dots, S_1^* A_k S_1)$ is generically identifying. Since $A_j, j\ge 2$ was generic, the matrices $S_1^* A_2 S_1,\dots, S_1^* A_k S_1$ are also generic, and independent of $U_1$ therefore they can be replaced anew by generic $A_2,\dots, A_k$. Repeating the argument $k$ times yields the claim. The proof for (ii) is analogous, noting that identifiability holds for generic $(A_1,\dots, A_k)$, but all $z$.
\end{proof}

In our terminology, Proposition~\ref{Prop:opencond} also implies that the behavior of random projectors is completely determined by their number, and no other properties. This motivates the following:
\begin{Def}\label{Def:thres}
Consider an arbitrary family of irreducible varieties $\calP_i, i\in \NN.$ We will denote
\begin{description}
   \item[(i)] in case of existence, the smallest number $k$ such that $(\calP_1,\dots, \calP_k)$ is generically identifying by $\lambda (\calP_1,\calP_2,\dots)$. We will denote the number, if clear from the context, by $\lambda (\calP)$, and call it the \emph{generic identifiability threshold}.
   \item[(ii)] in case of existence, the smallest number $k$ such that $(\calP_1,\dots, \calP_k)$ is completely identifying by $\kappa (\calP_1,\calP_2,\dots)$. We will denote the number, if clear from the context, by $\kappa (\calP)$, and call it the \emph{complete identifiability threshold}.
\end{description}
If $\calP_i=\calX$ for all $i$, for some variety $\calX$, we also write $\lambda(\calX)$ and $\kappa (\calX)$ instead of $\lambda(\calP_1,\calP_2,\dots)$ and $\kappa (\calP_1,\calP_2,\dots)$.
\end{Def}

\subsection{From Complex to Real Identifiability}
\label{sec:theorems.descent}
Before deriving identifiability statements in the given terminology, we briefly derive results which allow to return to the original phase retrieval problem~\ref{Prob:original}; that is, we state the principle of excluded middle for real measurements and signals. It implies that the conclusions of our main theorems~\ref{Thm:algsignalshort} and~\ref{Thm:algmeasureshort} hold for the non-algebraized, real formulation as well:

\begin{Prop}\label{Prop:proboneR}
Write $\calS_\RR:=\calS\cap (\RR^{n\times n})^\gamma$ and $\calP_\RR:=\left(\calP_1\cap (\RR^{n\times n})^\gamma\right)\times \dots \times \left(\calP_k\cap (\RR^{n\times n})^\gamma\right)$ for their real parts. Assume that $\calS$ and $\calP$ are observable over the reals (as defined in appendix~\ref{app:realcompgen}). Then, the following statements, about identifying signals $Z\in\calS$ from $\Tr(Z\cdot A_1),\dots, \Tr(Z\cdot A_1)$ hold:
\begin{description}
   \item[(i)] If $(A_1,\dots, A_k)\in\calP_\RR$ is not generically identifying (viewed as an element of $\calP$), then no signal $Z\in\calS_\RR$ can be perturbation-stably identified.
   \item[(ii)] If $(A_1,\dots, A_k)\in\calP_\RR$ is generically identifying (viewed as an element of $\calP$), then a generic signal $Z\in\calS_\RR$ can be perturbation-stably identified.
   \item[(iii)] If $(A_1,\dots, A_k)\in\calP_\RR$ is completely identifying (viewed as an element of $\calP$), then all signals $Z\in\calS_\RR$ can be perturbation-stably identified.
   \item[(iv)] If $\calP$ is not generically identifying, then no signal $Z\in\calS_\RR$ can be perturbation-stably identified by a generic $(A_1,\dots, A_k)\in\calP_\RR$.
   \item[(v)] If $\calP$ is generically identifying, then a generic signal $Z\in\calS_\RR$ can be perturbation-stably identified by a generic $(A_1,\dots, A_k)\in\calP_\RR$.
   \item[(vi)] If $\calP$ is completely identifying, then all signals $Z\in\calS_\RR$ can be perturbation-stably identified by a generic $(A_1,\dots, A_k)\in\calP_\RR$.
\end{description}
Here, in the statements above, generic means that all objects fulfill the statement, except possibly a Hausdorff measure zero set, where the Hausdorff measure is taken to be positive only on the highest dimensional components.\\

Furthermore, all statements hold when replacing $\calS_\RR$ by any positive measure subset $\calS'_\RR$ such that the Zariski closure of $\calS'_\RR$ is $\calS_\RR$, or replacing $\calP_\RR$ by any positive measure subset $\calP'_\RR$ such that the Zariski closure of $\calP'_\RR$ is $\calP_\RR$.
\end{Prop}
\begin{proof}
Statements (i) and (iv) follow already from the definitions, and Proposition~\ref{Prop:opencond}. The other numbered statements are implied by Theorem~\ref{Thm:genreal} in the appendix, noting that all properties above (or their negations) are algebraic. The last statement follows from the fact that if a set $V\subseteq \calS_\RR$ containing no, generic, or all elements of $\calS_\RR$, the set $V\cap \calS'_\RR$ contains no, generic, or all elements of $\calS'_\RR$, and the analogue for $\calP_\RR$ and $\calP'_\RR$.
\end{proof}

\begin{Rem}
There are several remarks to make:
\begin{description}
   \item[(i)] The assumption that $\calP$ and $\calS$ are observable over the reals are valid for our prototypical cases, see Proposition~\ref{Prop:irrvarieties}.
   \item[(ii)] An analogue of Proposition~\ref{Prop:proboneR} could be derived for the projections $A_i$ allowed to be complex and $z$ restricted to real signals or vice versa.
\end{description}
\end{Rem}

\subsection{Identifiability of Real Signals}

For the reals, there are natural lower bounds on the identifiability thresholds:

\begin{Prop}\label{Prop:lowerbound}
Consider identifiability from real signals $\calS=\{zz^\top, z\in\CC^n\}.$ For any family of irreducible varieties $\calP_i\subseteq \CC^{n\times n}, i\in \NN,$ with $n\ge 2$, it holds that $\kappa(\calP)\ge \lambda(\calP)$, and $\lambda(\calP)\ge n+1$.
\end{Prop}
\begin{proof}
If suffices to show that no $(A_1,\dots, A_n)\in \left(\CC^{n\times n}\right)^n$ can be stably generically identifying. We proceed by contradiction and assume the contrary. Proposition~\ref{Prop:algcritmatrices} then implies that $\left(\CC^{n\times n}\right)^n$ is generically identifying, so we may replace $A_1,\dots, A_n$ by a generic choice in $\left(\CC^{n\times n}\right)^n$. Fixing $z^\top A_1 z,\dots, z^\top A_n z$ yields $n$ equations on $z$, of degree $2$. By Bezout's theorem, and using that the $A_i$ are generic, those equations have $2^n$ solutions. Sign ambiguity leaves $2^{n-1}\gneq 1$ solutions, yielding a contradiction.
\end{proof}

Note that there are $(A_1,\dots, A_n)$ which are generically identifying but not stably so.

We now summarize some results which can be readily inferred from literature for real signals:

\begin{Thm}\label{Thm:knownthresreal}
Consider identifiability from real signals, corresponding to the complex signal variety $\calS_\rho=\{zz^\top, z\in\CC^n\}$, and projectors $\calP =\calS$. Then:
$$\lambda (\calP)=n+1,\quad\mbox{and}\quad \kappa(\calP) = 2n-1.$$
\end{Thm}
\begin{proof}
Note that once we have identifiability for signals $\calS'=\{zz^\top, z\in\RR^n\}$ and projectors $\calP'=\calS'$, we can use Proposition~\ref{Prop:opencond} to obtain the statement for the Zariski closure $\calS$ of $\calS'$ and $\calP$ of $\calP'$. So
$\lambda(\calP)\le n+1$ can be inferred from~\cite[Theorems~2.9]{Balan:2006fk}, and $\kappa(\calP) = 2n-1$ from~\cite[Theorem~2.2 and Corollary~2.7]{Balan:2006fk}. Combined with Proposition~\ref{Prop:lowerbound}, we obtain the statement.
\end{proof}

By virtue of Lemma~\ref{Lem:widen}, these results can immediately be broadened to include general linear projections, while Lemma~\ref{Lem:orthdescent} yields the case of orthogonal measurements:

\begin{Thm}\label{Thm:genlinthresreal}
Consider identifiability from real signals, corresponding to the complex signal variety $\calS=\{zz^\top, z\in\CC^n\}$, and the family $\calP_i =\{P^\top \cdot P\;:\;P\in\CC^{r_i\times n}\}, i\in\NN$ of projectors of potentially different ranks $r_i\ge 1$. Then:
$$
\lambda (\calP)=n+1,\quad\mbox{and}\quad \kappa(\calP) = 2n-1.
$$
The result remains unaltered if the projectors $\calP$ are restricted to be orthogonal.
\end{Thm}
\begin{proof}
$\lambda (\calP)\ge n+1$ is implied by Proposition~\ref{Prop:lowerbound}. $\kappa(\calP) \ge 2n-1$ is implied by Theorem~\ref{Thm:knownthresreal} and the definition of $\kappa$. Lower bounds $\lambda (\calP)\le n+1$ and $\kappa(\calP) \le 2n-1$ are implied by combining Theorem~\ref{Thm:knownthresreal} and Lemma~\ref{Lem:widen}. The statement for orthogonal projectors follows from Lemma~\ref{Lem:orthdescent}.
\end{proof}

Using the tools introduced in Section~\ref{sec:theorems.descent}, we obtain from this statement about the complexified problem one about the original phase retrieval problem for the reals:

\begin{Thm}\label{Thm:genlinreal}
Let $P_i\in\RR^{r_i\times n},1 \le i\le k$ be generic. Then, a generic signal $z\in\RR^n$ is identifiable from $b_i = \|P_i z \|^2, 1\le i\le k$ up to sign if and only if $k\geq n+1$. All signals $z\in\RR^n$ are identifiable from $b_i = \|P_i z \|^2, 1\le i\le k$ up to sign if and only if $k\geq 2n -1$.
The result remains unaltered if the projectors $P_i$ are restricted to be orthogonal.
\end{Thm}
\begin{proof}
Taking generic $P_i$ is equivalent to having generic symmetric measurements of rank $r_i$, by Proposition~\ref{Prop:genprojmat}. By the same argument as in the beginning of Lemma~\ref{Lem:orthdescent}, we can thus assume that we have generic $A_i$ of rank $r_i$. The statement is then implied by Theorem~\ref{Thm:genlinthresreal}~(i) and Proposition~\ref{Prop:proboneR}, noting that identifiability of $Z=zz^\top\in\RR^{n\times n}$ is equivalent to identifiability of $z$ up to sign.
\end{proof}
This solves the open problems (1-6).

\subsection{Identifiability of Complex Signals}

The case of complex phase recognition is somewhat analogous to the real one, while more technical due to the special structure of the matrices involved.

\begin{Prop}\label{Prop:lowerboundcomplex}
Consider identifiability from complex signals, corresponding to the complex signal variety $\calS_\iota=\{(xx^\top + yy^\top, yx^\top - xy^\top)\;:\;x,y\in\CC^n\}.$ For any family of irreducible varieties $\calP_i\subseteq \CC^{n\times n}\times \CC^{n\times n}, i\in \NN,$ with $n\ge 2$, it holds that $\kappa(\calP)\ge \lambda(\calP)$, and $\lambda(\calP)\ge 2n$.
\end{Prop}
\begin{proof}
Let $\phi$ be the forward map in Problem~\ref{Prob:inversecompalg}. It holds that $\dim \calS_\iota=2n-1$, therefore the fiber $\phi^{-1}(\phi(A_1,\dots, A_k,Z))$ can be finite only if $k\ge 2n-1$. Since $\calS_\iota$ is non-linear, it has degree strictly bigger than one, implying by Bezout's theorem that $\phi^{-1}(\phi(A_1,\dots, A_k,Z))$ is not finite for $k=2n-1$. Therefore, $\lambda(\calP)\ge 2n$.
\end{proof}

We now summarize some results which can be readily inferred from literature:

\begin{Thm}\label{Thm:knownthrescomplex}
Consider identifiability from complex signals $\calS=\{(xx^\top + yy^\top, yx^\top - xy^\top)\;:\;x,y\in\CC^n\},$
and the family $\calP=\calS$ Then:
$$\lambda (\calP)=2n,\quad\mbox{and}\quad \kappa(\calP) \le 4n-4.$$
\end{Thm}
\begin{proof}
Note that once identifiability for signals $\calS'=\{(xx^\top + yy^\top, yx^\top - xy^\top)\;:\;x,y\in\RR^n\},$ and projectors $\calP'=\calS'$ is established, we can use Proposition~\ref{Prop:opencond} to obtain the statement for the Zariski closure $\calS$ of $\calS'$ and $\calP$ of $\calP'$.
Thus, $\lambda(\calP)\le 2n$ can be inferred from~\cite[Theorems~3.4]{Balan:2006fk}; the inequality $\kappa(\calP) \le 4n-4$ can be obtained from~\cite[section~4]{Balan:2013fk}. Combined with Proposition~\ref{Prop:lowerboundcomplex}, this yields the statment.
\end{proof}

Again, Lemmata~\ref{Lem:widen} and~\ref{Lem:orthdescent} yield a statement for general projectors:

\begin{Thm}\label{Thm:genlinthrescomplex}
Consider identifiability from complex signals $\calS=\{(xx^\top + yy^\top, yx^\top - xy^\top)\;:\;x,y\in\CC^n\},$
and the family $\calP_i:=\{(Q^\top Q + S^\top S,Q^\top S - S^\top Q) \;:\;S,Q\in\CC^{r_i\times n}\}.$  Then:
$$\lambda (\calP)=2n,\quad\mbox{and}\quad \kappa(\calP) \le 4n-4.$$
The result remains unaltered if the projectors $\calP$ are restricted to be unitary.
\end{Thm}
\begin{proof}
$\lambda (\calP)\ge 2n$ is implied by Proposition~\ref{Prop:lowerboundcomplex}. Lower bounds $\lambda (\calP)\le 2n$ and $\kappa(\calP) \le 4n-4$ are implied by combining Theorem~\ref{Thm:knownthrescomplex} and Lemma~\ref{Lem:widen}. The statement for unitary projectors follows from Lemma~\ref{Lem:orthdescent}.
\end{proof}

Now we can combine the algebraic results over the complex numbers in Theorem~\ref{Thm:genlinthrescomplex}, with the connection to the reals given in Proposition~\ref{Prop:proboneR} to deduce the identifiability theorems for the original (non-algebraized) complex phase retrieval problem:

\begin{Thm}\label{Thm:genlincomp}
Let $P_i\in\CC^{r_i\times n},1 \le i\le k$ be generic. Then, a generic signal $z\in\CC^n$ is identifiable from $b_i = \|P_i z \|^2, 1\le i\le k$ up to phase if and only if $k\geq 2n$. All signals $z\in\CC^n$ are identifiable from $b_i = \|P_i z \|^2, 1\le i\le k$ up to phase if $k\geq 4n-4$. The result remains unaltered if the projectors $P_i$ are restricted to be unitary.
\end{Thm}
\begin{proof}
Taking generic $P_i$ is equivalent to having generic symmetric measurements of rank $r_i$, by Proposition~\ref{Prop:genprojmat}. By the same argument as in the beginning of Lemma~\ref{Lem:orthdescent}, we can thus assume that we have generic $A_i$ of rank $r_i$. The statment is then implied by Theorem~\ref{Thm:genlinthrescomplex}~(i) and Proposition~\ref{Prop:proboneR}, noting that identifiability of $(X,Y)$ is equivalent to identifiability of $X+\iota Y=zz^*\in\CC^{n\times n}$, up to phase.
\end{proof}
This solves problem (7), and problems (1-6) for unitary projection matrices.

We would like to remark that Theorems~\ref{Thm:genlinthresreal} and~\ref{Thm:genlinthrescomplex}, together with Proposition~\ref{Prop:proboneR}, not only yield the statements for general linear projections, but in fact state that any bound which can be proved for the rank one projectors directly extends to any (irreducible) set of linear projectors containing those. In particular, if one succeeds in proving a lower bound for complete identifiability for complex phase retrieval, the same bound automatically holds for linear projectors of arbitrary rank.

\subsection{Generic Measurements are Crucial for the Identifiability Thresholds}\label{sec:theorems.ramex}

The identifiability thresholds are defined by means of generic measurements. This section is dedicated to verify there are specific measurements such that signal recovery is possible below the identifiability threshold. In fact, we shall provide an example of $n$ measurements in the real and $2n-1$ measurements in the complex case, so that generic signals are uniquely determined up to their sign and phase, respectively.
\begin{Ex}\label{Ex:ramex}
Consider the standard orthonormal basis vectors $\{e_i\}_{i=1}^n$ of $\RR^{n}$, and define $A_j:=e_1e_1^\top+e_je_1^\top+e_1e_j^\top$, $j=1,\ldots,n$. We observe that, for all $Z=zz^\top$, $z\in\RR^n$,
$$b_1=|z_1|^2,\quad\mbox{and}\quad b_j=\Tr(Z\cdot A_j)=|z_1|^2+2z_1z_j,\quad j=2,\ldots,n.$$
For a generic signal $z$, we can always assume $z_1\neq 0$. Since we must recover $z$ only up to its sign, we may as well assume that the measurement $b_1$ determines $z_1$ and aim to reconstruct $z_2,\dots,z_n$ exactly. A mere reformulation of the above equation then yields
$z_j = \frac{b_j-z_1^2}{2z_1},$
for $j=2,\ldots,n$. Thus, the $n$ measurements $\{A_j\}_{j=1}^n$ uniquely determine all signals $z\in\RR^n$ with $z_1\neq 0$ up to a global sign, hence generic signals.
\end{Ex}
A similar example can be derived for complex signals and $2n-1$ measurements:
\begin{Ex}\label{EX:2b}
We choose measurement matrices $A_1:= e_1e_1^\top$ and $A_j:=e_1e_1^\top+e_je_1^\top+e_1e_j^\top$ as well as $\tilde{A}_j:=e_1e_1^\top+\iota e_je_1^\top-\iota e_1e_j^\top$, $j=2,\ldots,n$. This choice yields the measurements
 \begin{equation*}
 b_1=|z_1|^2,\quad\mbox{and}\quad b_j=|z_1|^2+2\Re(\overline{z}_1 z_j),\quad c_j=|z_1|^2+2\Im(\overline{z}_1 z_j),\quad j=2,\ldots,n.
 \end{equation*}
Without loss of generality, we can choose $z_1$ real-valued and genericity allows to assume $z_1\neq 0$, so that we obtain
 \begin{equation*}
\Re(z_j) =  \frac{b_j-z_1^2}{2z_1},\quad \Im(z_j) =  \frac{c_j-z_1^2}{2z_1},\quad j=2,\ldots,n.
 \end{equation*}
Thus, the $2n-1$ measurements determine a generic $z$ up to a global phase factor.
 \end{Ex}

The matrices $A_2,\dots, A_n$ in Example~\ref{Ex:ramex} are all rank two, only $A_1$ is of rank one, similarly in Example \ref{EX:2b};  we would like to note that such an example does not exist for pure rank one measurements, that is, $A_i\in\calP_\rho(1)$ or $A_i\in\calP_\CC(1)$. Namely, in the case of real signals, a rank one measurement regime of size $n$ has either linearly dependent row-spans; otherwise, it is equivalent, by applying a linear transformation to $z$, to the measurements $A_j=e_je_j^\top$, which generically has $2^n$ distinct solutions. Thus, it is generically non-identifying in both cases. The argument for complex signals is similar. Therefore, it is highly crucial for the perturbation results in~\cite{Balan:2013fk} that the measurement scheme is rank one.

\section{A Deterministic Inversion Formula}\label{sec:algorithm}

\subsection{Phase Retrieval as Ideal Regression}
\label{sec:inversion.idreg}

In this section, we will show that the phase retrieval problem is a special case of an algebraic estimation problem, called ideal regression. This means that not only is the solvability and identifiability of the problem determined by algebraic invariants, such as $n,k,$ or the kind of projectors, but that it is - in principle - also accessible to algorithmical estimation tools from approximate algebra, such as those presented in \cite{Kiraly:2012fk}, yielding explicit and deterministic inversion formulae not only for $k=\Omega (n^2)$, but directly at the identifiability threshold $k\ge n+1$.

The reformulation of the phase retrieval as an algebraic estimation problem bears similarities to the algebraization in Section~\ref{sec:theorems.algpr}. The major idea consist of converting the observation into polynomials, which are then manipulated to obtain the solution.

Assume we are in the case of the real phase recognition problem, wanting to identify a signal $z\in\RR^n$. Then, let $X = (X_1,\dots, X_n)$ be a vector of formal variables. The $k$ projection matrices $P_i$ give rise to $k$ polynomials $$p_i(X_1,\dots, X_n) = X^\top A_i X - b_i$$
in the variables $X_j$, with $A_i= P_i^\top P_i$, such that, after substitution, we have $p_i(z)=0$. By definition the polynomials $p_i$ are contained in the ideal $\mathcal{I}:=\Id(z)\subseteq \CC [X_1,\dots, X_n]$. Thus, the estimation problem becomes, for the real phase recognition problem:\\

\begin{Prob}\label{Prob:idregreal}
Let $z\in\RR^n$ be unknown, let $\fraks=\left\langle X_1-z_1,\dots, X_n-z_n\right\rangle \in \CC [X_1,\dots, X_n]$. Let $p_1,\dots, p_k\in \mathcal{I}$ be known polynomials, of the form $p_i(X_1,\dots, X_n) = X^\top A_i X - b_i$, where $b_i=z^\top A_i z - b_i$.\\
Then, reconstruct $\fraks$, or equivalently, $z$, from the polynomials $p_1,\dots, p_k, 1\le i\le k$.\\
\end{Prob}

What at first seems like a mere reformulation, contains the gist of the algebraic ideal regression method: instead of fitting a loss function or performing optimization on $z$, or taking the $b_i,P_i$ as an input, we try to obtain the solution from manipulating the polynomials $p_i$ as symbolic objects in their own right. Again, we note that we are working over the complex numbers in the polynomial ring $\CC [X_1,\dots, X_n]$, similarly to the algebraization; we will again show that this is no major problem, from an algorithmic aspect.\\

The complex case is slightly different but can be treated similarly. Here, let $X = (X_1,\dots, X_n)$ and $Y= (Y_1,\dots, Y_n)$ be vectors of formal variables, and let $P_i = Q_i + \iota\cdot S_i$ with $Q_i,S_i\in\mathbb{R}^{m\times n}$. The projections give rise to $k$ polynomials
\begin{equation*}
p_i =  (X,Y)^\top\!\! \begin{pmatrix}
Q_i^\top Q_i + S_i^\top S_i & S_i^\top Q_i - Q_i^\top S_i\\
Q_i^\top S_i - S_i^\top Q_i & Q_i^\top Q_i + S_i^\top S_i
\end{pmatrix} \!\!\begin{pmatrix}X\\ Y
\end{pmatrix},
\end{equation*}
and those are, similar to the real case, contained in the ideal\\
$\mathcal{I}:=\fraks ((X,Y)-\tilde{z})\subseteq \CC [X_1,\dots, X_n,Y_1,\dots, Y_n]$, where $\tilde{z}= (\Re z, \Im z)\in \mathbb{R}^{2n}$. So the estimation problem is, in the complex case:\\

\begin{Prob}\label{Prob:idregcomplex}
Let $\tilde{z}\in\RR^{2n}$ be an unknown point, let \\
$\fraks=\left\langle X_1-\Re z_1, Y_1-\Im z_1,\dots, X_n-\Re z_n, Y_n-\Im z_n\right\rangle\subseteq \CC [X_1,\dots, X_n,Y_1,\dots, Y_n]$.\\
Let $p_1,\dots, p_k\in \fraks$ be known polynomials, of the form as above. Reconstruct $\fraks$, or equivalently $\tilde{z},$ from the $p_1,\dots, p_k, 1\le i\le k$.
\end{Prob}

Note that the ideal regression formulation of phase retrieval Problem~\ref{Prob:idregcomplex} differs fundamentally from the algebraized inverse problem version given in Problem~\ref{Prob:inversecompalg}, since in ideal regression, we split real and complex parts of the formal variables, whereas in the algebraization, we split real and complex parts of the matrices involved. Still, both problems are intrinsically related, and can be considered, in a certain sense, as each other's duals.

\subsection{An Inversion Formula with Ideal Regression}
\label{sec:inversion.idreg}

We describe how the ideal regression formulation of the phase retrieval problem~\ref{Prob:idregreal} can be solved by an approximate algebraic algorithm; we focus on the real case.
%
%
If $k\ge {n+1 \choose 2}$, there exist explicit inversion formulae in which one computes an approximation for $\Tr (A_i zz^\top) = b_i$, which is now considered as a linear system of $k$ equations in the ${n+1\choose 2}$ unknowns $zz^\top$; this can be written as pseudo-inverting a matrix which has one row per $A_i$, and noise stability can be achieved by regularization, or by performing a singular value decomposition. On the other hand if $k\lneq {n+1\choose 2}$, such a direct approach will not work.

However, it is nevertheless possible to construct an explicit deterministic inversion formula, readily providing answers at the identifiability threshold $k\ge n+1$, and which is numerically stable. The main idea is to use an ideal regression algorithm, namely Algorithms~1 and~2 in~\cite{Kiraly:2012fk}; the ideal $\fraks$ we wish to estimate in our case is linear, namely $\fraks = \langle X_1 - z_1, \dots, X_n-z_n\rangle$, and the input polynomials are of degree two, contained in $\fraks$. Since $\fraks$ is inhomogenous, Algorithm 1~in~\cite{Kiraly:2012fk} will output the homogenous part of $\fraks$, namely $\fraks_h=\fraks\cap \langle X_1,\dots, X_n\rangle$ which is also linear, and can be used to estimate $z$.

Instantiating Algorithm~1 in~\cite{Kiraly:2012fk} with $D=n, d=1$, and polynomials $f_i:= p_i/b_i - \overline{p}, 1\le i\le k-1$, where $\overline{p} = \sum_{i=1}^k p_i/b_i$, yields an estimate for generators $\ell_1,\ell_{n-1}$ of $\fraks_h$. The signal $z$ fulfills $\ell_i(z)=0$, therefore $z$ is orthogonal to the coefficient vectors of the $\ell_i$ and can be determined up to a scalar multiple $z'=\alpha z$ from the $\ell_i$. Thus, $z$ can be determined by setting $z:=z'/\alpha$ where $\alpha$ can be estimated as $\alpha:=\exp \left( \sum_{i=1}^k \log\left((z')^\top A_i z\right) - \log b_i\right)$. We will refer to this strategy as the ``explicit inversion'' in the experiments section.

We refrain from actually explaining in detail how Algorithm~1 in~\cite{Kiraly:2012fk} works, or from stating the algorithm itself, due to the amount of notational overhead which would be needed, and refer the reader to the original paper instead. We want to stress that Algorithm~1 is deterministic and numerically stable, therefore it yields a potentially explicit and regularizable inversion formula for the phase recognition problem.

\section{Experiments}\label{sec:experiments}
\begin{figure}
\centering
\subfigure[$n=6$]{
\includegraphics[width=.3\textwidth]{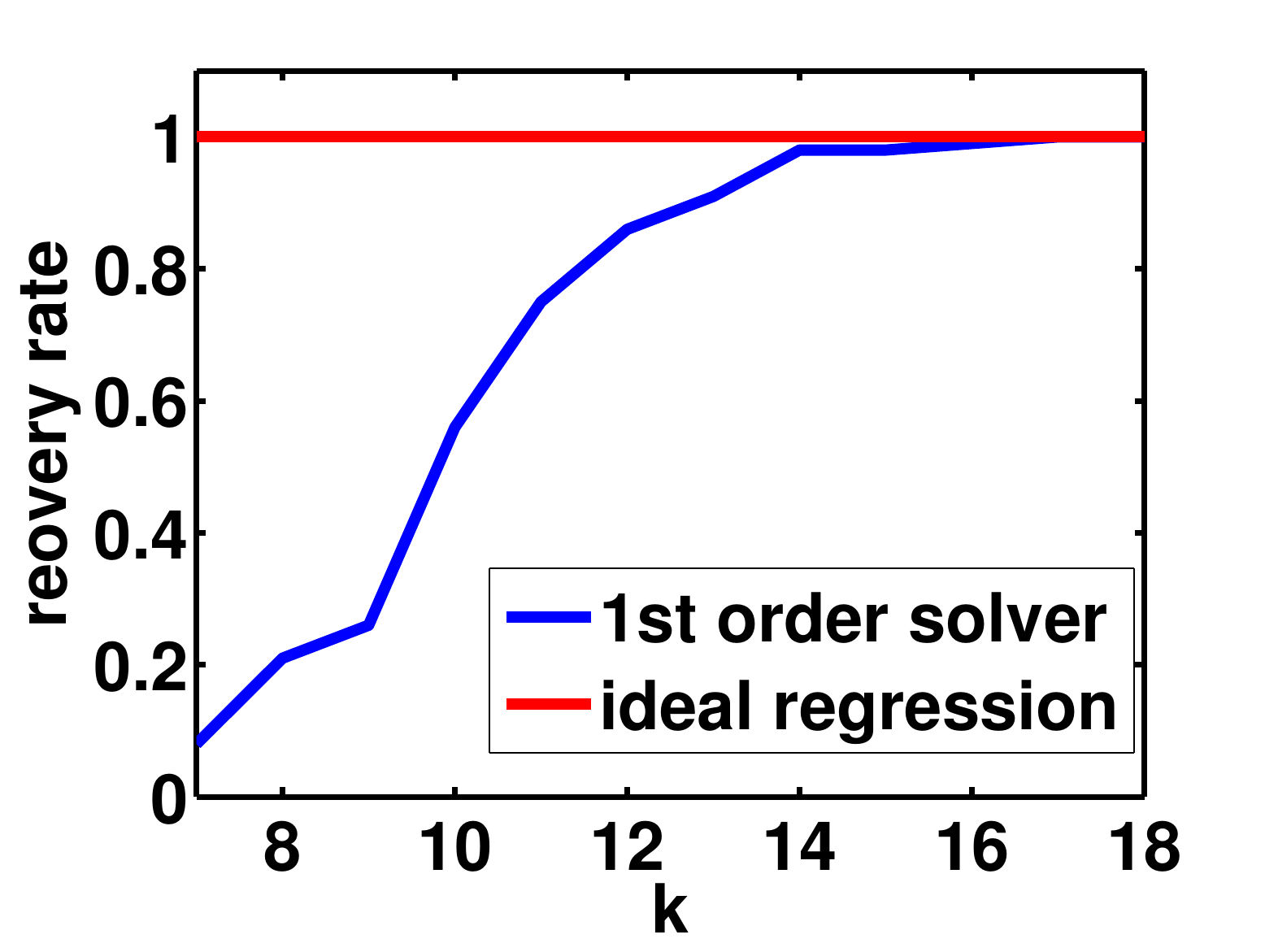}
}
\subfigure[$n=8$]{
\includegraphics[width=.3\textwidth]{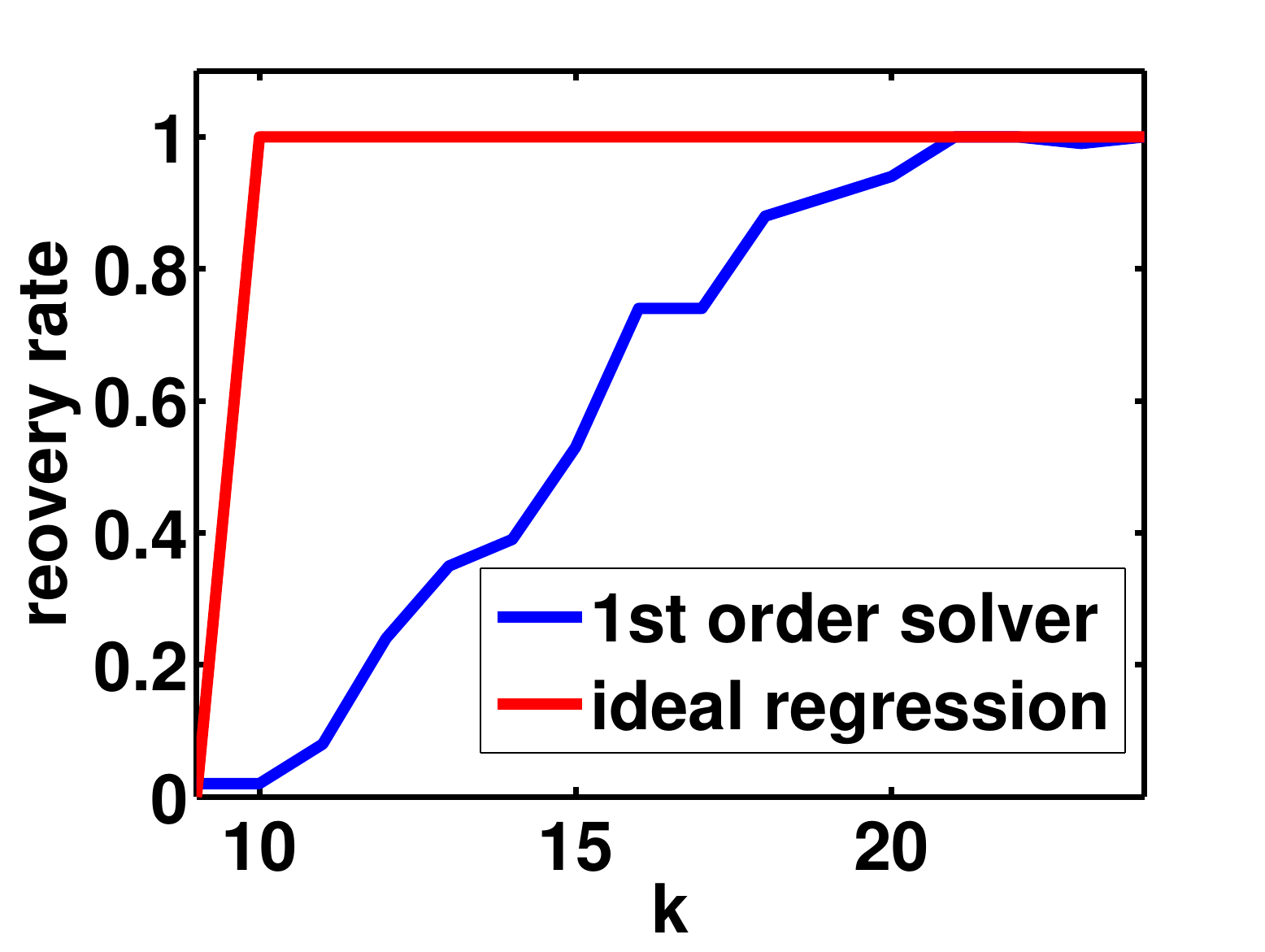}
}
\subfigure[$n=10$]{
\includegraphics[width=.3\textwidth]{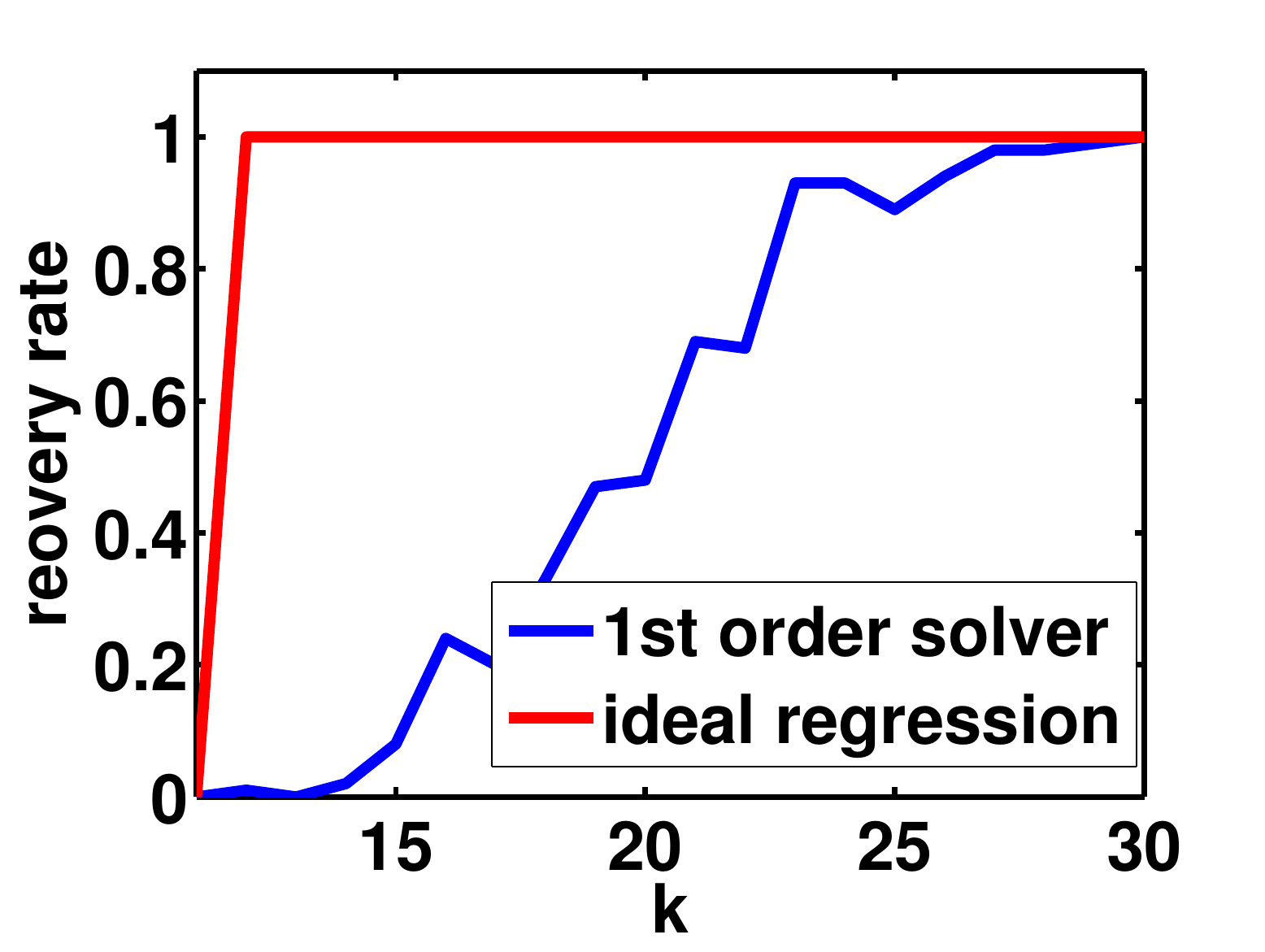}
}
\caption{Recovery rates averaged over $100$ repeats without any noise for ideal regression and for a first order solver in PhaseLift.}\label{fig:exp:plot 1}
\end{figure}

\begin{figure}
\centering
\subfigure[Noise level $\sigma=10^{-2}$.]{
\includegraphics[width=.4\textwidth]{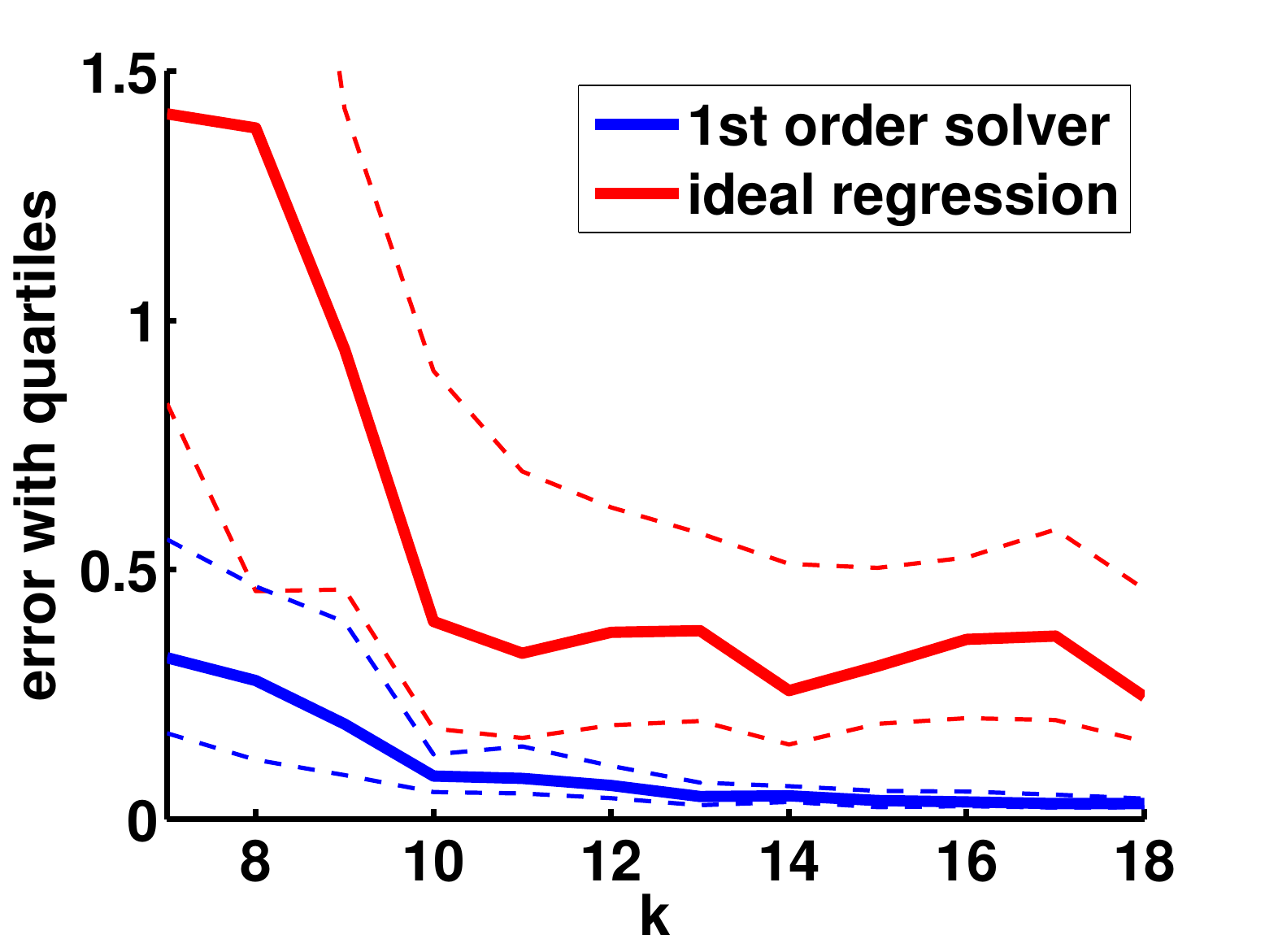}\label{subfig:a}
}
\;
\subfigure[Noise level $\sigma=10^{-4}$.]{
\includegraphics[width=.4\textwidth]{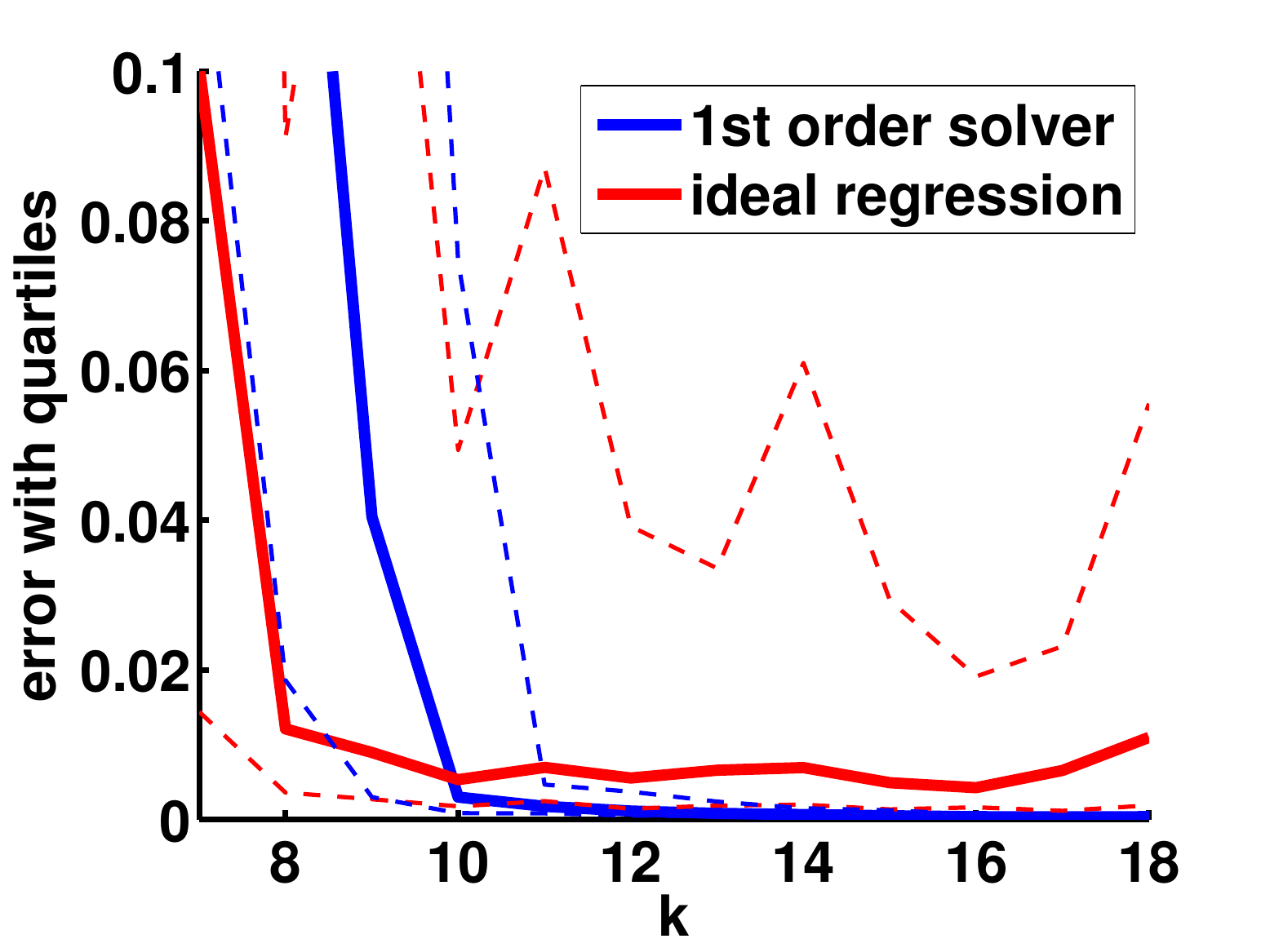}\label{subfig:b}
}
\caption{Mean squared error for $n=6$ and quartiles for $100$ repeats.}\label{fig:exp:plot 2}
\end{figure}

In this section we provide few numerical experiments illustrating that generic real signals can be identified from few generic magnitude measurements by using the inversion formula obtained from ideal regression as outlined in section~\ref{sec:inversion.idreg}. We also include a few comparisons to an alternative method. Classical phase retrieval algorithms such as Gerchberg/Saxton \cite{Gerchberg:1972kx} and Fienup's alternatives \cite{Fienup:1982vn} are customized to Fourier measurements, hence are also limited to this setting. An approach that can deal with generic measurements is PhaseLift~\cite{Candes:uq}, which is based on finding the feasible point of a semidefinite program and is proposed to be solved using first order methods. The theoretical results in \cite{Candes:uq} are asymptotic in the ambient dimension $n$ and no success guarantees are derived for fixed $n$. Nonetheless, PhaseLift is known to be quite successful and very robust against noise in practise. The complexity of ideal regression causes limits in the number of measurements that can be dealt with in practise, while it yields an explicit reconstruction formula. We shall study the performance of ideal regression and PhaseLift for few measurements.

In the numerical experiments, we choose the signal $x$ uniformly distributed on the sphere.  Measurements are performed by orthogonal rank-1 projectors, also uniformly distributed (according to the standard Haar measure on this set), and we deal with corrupted measurements $\tilde{b}=b+\eta$, where $\eta$ is Gaussian white noise of variance $\sigma$. The outcome of performance comparisons between ideal regression and PhaseLift very much depend on the noise level. If measurements are exact, then ideal regression yields signal recovery for generic $n+1\leq k\leq 3n$ measurements, a range, in which PhaseLift performs rather poorly, see Fig.~\ref{fig:exp:plot 1} for $n=6,8,10$. For inexact yet still very accurate measurements, in other words very low noise levels ($\sigma\approx 10^{-4}$), ideal regression still outperforms PhaseLift when the number of measurements is close to the threshold $n+1$, see Fig.~\ref{subfig:b}, with a comparable accuracy for higher noise levels ($\sigma\approx 10^{-2}$), cf.~Fig.~\ref{subfig:a}. Nonetheless, it must be mentioned that with slightly larger and hence more common noise levels, especially when the number of measurements increases, then PhaseLift is eventually to be favored since error rates are then significantly smaller than within ideal regression. It is interesting to note that ideal regression performs well close to the identifiability threshold $k=n+1$, whereas PhaseLift yields more accurate estimates as the number of samples increases.

\section*{Acknowledgements}

ME is funded by the Vienna Science and Technology Fund (WWTF) through project VRG12-009. FK is supported by Mathematisches Forschungsinstitut Oberwolfach (MFO).

\bibliographystyle{plainnat}
\bibliography{bibtex_phase}

\appendix
\newpage
\section{Algebraic Geometry Fundamentals}
\subsection{Algebraic Geometry Glossary}\label{app:algebraic-glossary}
We briefly give a glossary of algebraic terms used in the main corpus. Let $\KK=\RR$ or $\KK=\CC$.
\begin{Def}
A set $\calX\subseteq \KK^n$ is called \emph{algebraic variety} if there are polynomials $f_1,\dots, f_n$ variables such that
$$\calX=\{x\in\KK^n\;:\; f_1(x)=\dots = f_n(x)=0\}.$$
\end{Def}

\begin{Def}
The \emph{Zariski topology} on $\KK^n$ is the induced topology in which algebraic varieties are open. That is, \emph{Zariski closed} sets being finite unions of algebraic varieties, and \emph{Zariski open} sets the complement. The Zariski topology on some variety $\calX$ is the induced relative topology.
\end{Def}

\begin{Def}
An algebraic variety $\calX\subseteq \KK^n$ is called \emph{irreducible} if can not be written as a proper union of algebraic varieties. That is, if $\calX=\calX_1\cup \calX_2$ for algebraic varieties $\calX_1,\calX_2$, then $\calX_1\subseteq \calX_2$ or $\calX_2\subseteq \calX_1$.
\end{Def}

\begin{Def}\label{Def:morphism}
Let $f_1,\dots, f_m$ be polynomials in $n$ variables, let $\calX\subseteq \KK^n$ and $\calY\subseteq \KK^m$ be algebraic varieties. A mapping $$\phi:\calX\rightarrow \calY,\quad x\mapsto (f_1(x),\dots, f_m(x))$$
is called \emph{algebraic map} or \emph{morphism of algebraic varieties}.
\end{Def}

\begin{Def}
A morphism of algebraic varieties, as above, is called \emph{unramified at} $x\in\calX$ and \emph{unramified over} $\phi (x)\in\calY$, if there is a Borel-open neibhourhood $U\subseteq \KK^n$ (cave: not $U\subseteq \calX$), with $x\in U$ such that for all $z\in U$, if holds that $\card{\phi^{-1}\phi(x)}=\card{\phi^{-1}\phi(z)}$. If $\calX$ is irreducible, $\phi$ is called \emph{generically unramified} if the points $x\in\calX$ at which $\phi$ is ramified are contained in a proper Zariski closed subset of $\calX$.
\end{Def}

\begin{Def}
A generically unramified morphism, as above, with $\calX$ and $\calY$ irreducible, is called \emph{birational} if there is a proper Zariski closed subset $\calZ$ of $\calX$ such that $f$, restricted to $\calX\setminus\calZ$, is bijective.
\end{Def}

\subsection{Open Conditions and Generic Properties of Morphisms}\label{app:algebraic-geometry}
In this section, we will summarize some algebraic geometry results used in the main corpus. The following results will always be stated for algebraic varieties over $\CC$.

\begin{Prop}\label{Prop:irreducible}
Let $f:\calX\rightarrow \calY$ be a morphism of algebraic varieties (over any field). Then, if $\calX$ is irreducible, so is $f(\calX)$. In particular, if $f$ is surjective, and $\calX$ is irreducible, then $\calY$ also is.
\end{Prop}
\begin{proof}
This is classical; suppose the converse, that is, $f(\calX)=\calZ_1\cup \calZ_2$ is a proper union of algebraic sets. Then, using that $f$ is algebraic, and therefore continuous in the Zariski topology, it follows that $\calX$ is a proper union $\calX=f^{-1}(\calZ_1)\cup f^{-1}(\calZ_2)$ of algebraic sets. This contradicts $\calX$ being irreducible, proving the statement by contraposition.
\end{proof}

\begin{Thm}
Let $f:\calX\rightarrow \calY$ be a morphism of algebraic varieties. The function
$$\calY\rightarrow\NN,\quad y\mapsto \dim f^{-1}(y)$$
is upper semicontinuous in the Zariski topology.
\end{Thm}
\begin{proof}
This follows from~\cite[Th{\'e}or{\`e}me~13.1.3]{EGAIV.3}.
\end{proof}

\begin{Prop}\label{Prop:flatgen}
Let $f:\calX\rightarrow \calY$ be a morphism of algebraic varieties, with $\calY$ be irreducible. Then, there is an open dense subset $V\subseteq \calY$ such that $f: U\rightarrow V$, where $U=f^{-1}(V)$, is a flat morphism.
\end{Prop}
\begin{proof}
This follows from~\cite[Th{\'e}or{\`e}me~6.9.1]{EGAIV.2}.
\end{proof}

\begin{Thm}\label{Thm:openconds}
Let $f:\calX\rightarrow \calY$ be a morphism of algebraic varieties. Let $d,\nu\in\NN$. Then, the following are open conditions for $y\in \calY$; that is, the sets $\{y\in\calY\;:\;\mbox{condition (*) holds for}\; y\}$ is a Zariski open subset of $\calY$.
\begin{description}
\item[(i)] $\dim f^{-1}(y)\le d$.
\item[(ii)] $f$ is unramified over $y$.
\item[(iii)] $f$ is unramified over $y$, and the number of irreducible components of $f^{-1}(y)$ equals $\nu$.
\end{description}
In particular, if $f$ is surjective, then the following is an open property as well:
\begin{description}
\item[(iv)] $f$ is unramified over $y$, and $\card {f^{-1}(y)}=\nu$.
\end{description}
\end{Thm}
\begin{proof}
(i) follows from~\cite[Corollaire~6.1.2]{EGAIV.2}.\\
(ii) follows from~\cite[Th{\'e}or{\`e}me~12.2.4(v)]{EGAIV.3}.\\
(iii) follows from~\cite[Th{\'e}or{\`e}me~12.2.4(vi)]{EGAIV.3}.\\
(iv) follows from (i), applied in the case $\dim f^{-1}(y)\le 0$ which is equivalent to $\dim f^{-1}(y) = 0$ due to surjectivity of $f$, and (iii).
\end{proof}

\begin{Cor}\label{Cor:genericprops}
Let $f:\calX\rightarrow \calY$ be a generically unramified and surjective morphism of algebraic varieties, with $\calY$ be irreducible. Then, there are unique $d,\nu\in\NN$ such that the following sets are Zariski closed, proper subsets of $\calY$ (and therefore Hausdorff zero sets):
\begin{description}
\item[(i)] $\{y\;:\;\dim f^{-1}(y)\neq d\}$
\item[(ii)] $\{y\;:\;f\;\mbox{is ramified at}\;y\}$
\item[(iii)] $\{y\;:\;f\;\mbox{is ramified at}\;y\}\cup\{y\;:\;\card {f^{-1}(y)}\neq \nu\}$
\end{description}
\end{Cor}
\begin{proof}
This is implied by Theorem~\ref{Thm:openconds}~(i), (ii) and (iii), using that a non-zero open subset of the irreducible variety $\calY$ must be open dense, therefore its complement in $\calY$ is a closed and a proper subset of $\calY$.
\end{proof}

\begin{Prop}\label{Prop:opencert}
Let $f:\calX\rightarrow\calY$ be a morphism of algebraic varieties, with $\calY$ irreducible. Then, the following are equivalent:
\begin{description}
\item[(i)] $f$ is unramified over $y$ and $\card{f^{-1}(y)} = \nu$.
\item[(ii)] There is a Borel open neighborhood $U\subseteq \calY$ of $y\in U$, such that $f$ is unramified over $U$ and $\card{f^{-1}(z)} = \nu$ for all $z\in U$.
\item[(iii)] There is a Zariski open neighborhood $U\subseteq \calY$ of $y\in U$, dense in $\calY$, such that $f$ is unramified over $U$ and $\card{f^{-1}(z)} = \nu$ for all $z\in U$.
\end{description}
\end{Prop}
\begin{proof}
The equivalence is implied by Corollary~\ref{Cor:genericprops} and the fact that $\calY$ is irreducible. Note that either condition implies that $f$ is generically unramified due to Theorem~\ref{Thm:openconds}~(ii) and irreducibility of $\calY$.
\end{proof}

\subsection{Real versus Complex Genericity}\label{app:realcompgen}

We derive some elementary results how generic properties over the complex and real numbers relate. While some could be taken for known results, they appear not to be folklore - except maybe Lemma~\ref{Lem:cmprl}. In any case, they seem not to be written up properly in literature known to the authors.

\begin{Def}
Let $\calX\subseteq \CC^n$ be a variety. We define the \emph{real part} of $\calX$ to be $\calX_\RR:=\calX\cap \RR^n$.
\end{Def}

\begin{Lem}\label{Lem:cmprl}
Let $\calX\subseteq \CC^n$ be a variety. Then, $\dim \calX_\RR\le \dim \calX$, where $\dim \calX_\RR$ denotes the Krull dimension of $\calX_\RR$, regarded as a (real) subvariety of $\RR^n$, and $\dim \calX$ the Krull dimension of $\calX$, regarded as subvariety of $\CC^n$.
\end{Lem}
\begin{proof}
Let $k=n - \dim\calX$. By~\cite[section~1.1]{Mumford}, $\calX$ is contained in some complete intersection variety $\calX'=\Van (f_1,\dots, f_k)$. That is $(f_1,\dots, f_k)$ is a complete intersection, with $f_i\in\CC[X_1,\dots, X_n]$ and $\dim\calX'=\dim\calX$, such that $f_i$ is a non-zero divisor modulo $f_1,\dots, f_{i-1}$. Define $g_i:=f_i\cdot f_i^*$, one checks that $g_i\in\RR[X_1,\dots, X_n]$, and define $\calY:=\Van (g_1,\dots, g_k)$ and $\calY_\RR:=\calY\cap \RR^n$. The fact that $f_i$ is a non-zero divisor modulo $f_1,\dots, f_{i-1}$ implies that $g_i$ is a non-zero divisor modulo $g_1,\dots, g_{i-1}$; since $g_i\cdot h \cong 0$ modulo $g_1,\dots, g_{i-1}$ implies $f_i\cdot (h\cdot f_i^*) \cong 0$ modulo $f_1,\dots, f_{i-1}$. Therefore, $\dim\calY_\RR \le \dim\calX$; by construction, $\calX'\subseteq \calY$, and $\calX\subseteq \calX'$, therefore $\calX_\RR\subseteq \calY_\RR$, and thus $\dim \calX_\RR\le \dim \calY_\RR$. Combining it with the above inequality yields the claim.
\end{proof}

\begin{Def}
Let $\calX\subseteq \CC^n$ be a variety. If $\dim \calX = \dim \calX_\RR$, we call $\calX$ \emph{observable over the reals}. If $\calX$ equals the (complex) Zariski-closure of $\calX_\RR$, we call $\calX$ \emph{defined over the reals}.
\end{Def}

\begin{Prop}\label{Prop:realdefobs}
Let $\calX\subseteq \CC^n$ be a variety.
\begin{description}
\item[(i)] If $\calX$ is defined over the reals, then $\calX$ is also observable over the reals.
\item[(ii)] The converse of (i) is false.
\item[(iii)] If $\calX$ irreducible and observable over the reals, then $\calX$ is defined over the reals.
\end{description}
\end{Prop}
\begin{proof}
(i) Let $k=n - \dim\calX_\RR$. By~\cite[section~1.1]{Mumford}, $\calX_\RR$ is contained in some complete intersection variety $\calX'=\Van (f_1,\dots, f_k)$, with $f_i\in\RR[X_1,\dots, X_n]$ a complete intersection. By an argument, analogous to the proof of Lemma~\ref{Lem:cmprl}, one sees that the $f_i$ are a complete intersection in $\CC [X_1,\dots, X_n]$ as well. Since the Zariski-closure of $\calX_\RR$ and $\calX$ are equal, it holds that $f_i\in \Id (\calX)$. Therefore, $\calX\subseteq \Van (f_1,\dots, f_k)$, which imples $\dim\calX\le n-k$, and by definition of $k$, as well $\dim\calX\le\dim \calX_\RR$. With Lemma~\ref{Lem:cmprl}, we obtain $\dim\calX_\RR = \dim\calX$, which was the statement to prove.\\
(ii) It suffices to give a counterexample: $\calX = \{1,i\}\subseteq \CC$. Alternatively (in a context where $\varnothing$ is not a variety) $\calX = \{(1,x)\;:\; x\in \CC\}\cup \{(i,x)\;:\; x\in \CC\}\subseteq \CC^2$.\\
(iii) By definition of dimension, Zariski-closure preserves dimension. Therefore, the closure $\overline{\calX_\RR}$ is a sub-variety of $\calX$, with $\dim \overline{\calX_\RR} = \dim\calX$. Since $\calX$ is irreducible, equality $\overline{\calX_\RR} = \calX$ must hold.
\end{proof}

\begin{Thm}\label{Thm:genreal}
Let $\calX\subseteq \CC^n$ be an irreducible variety which is observable over the reals, let $\calX_\RR$ be its real part. Let $P$ be an algebraic property. Assume that a generic $x\in\calX$ is $P$. Then, a generic $x\in \calX_\RR$ has property $P$ as well.
\end{Thm}
\begin{proof}
Since $P$ is an algebraic property, the $P$ points of $\calX$ are contained in a proper sub-variety $\calZ\subseteq \calX$, with $\dim\calZ\lneq \dim\calX$. Since $\calX$ is observable over the reals, it holds $\dim\calX=\dim\calX_\RR$. By Lemma~\ref{Lem:cmprl}, $\dim\calZ_\RR\le \dim\calZ$. Putting all (in-)equalities together, one obtains $\dim\calZ_\RR\lneq \dim\calX_\RR$. Therefore, the $\calZ_\RR$ is a proper sub-variety of $\calX_\RR$; and the $P$ points of $\calX_\RR$ are contained in it - this proves the statement.
\end{proof}

\section{Results on Phase Retrieval}
\subsection{Properties of the Forward Map}\label{app:forwardmap}
In this section we will check that the technical assumptions hold in the case of the relevant examples. We start with introducing notation for two maps which relate the signal/measurement varieties to projection matrices:
\begin{Not}
In the following, we will denote
\begin{align*}
\Upsilon:& \CC^{r\times n}\times \CC^{r\times n}\rightarrow \calP_\rho(r),\quad (Q,S)\mapsto Q^\top S,\\
\Upsilon_\CC:& \CC^{r\times n}\times \CC^{r\times n}\rightarrow \calP_\CC(r),\quad (Q,S)\mapsto (Q^\top Q + S^\top S,Q^\top S - S^\top Q).
\end{align*}
\end{Not}
The maps $\Upsilon$ and $\Upsilon_\CC$ can be seen to be surjective; as an immediate consequence of this fact, we can relate genericity of projections to genericity of measurement matrices:

\begin{Prop}\label{Prop:genprojmat}
Let $P,Q\in \CC^{r\times n}$ be generic matrices. Then: 
\begin{description}
\item[(i)] $\Upsilon (P,Q)$ resp.~$\Upsilon_\CC (P,Q)$ are generic inside $\calP_\rho(r)$ resp.~$\calP_\CC(r)$
\item[(ii)] $\Upsilon (P,P)$ resp.~$\Upsilon_\CC (P,P^*)$ are generic Hermitian matrices inside $\calP_\rho(r)$ resp.~$\calP_\CC(r)$
\end{description}
\end{Prop}
\begin{proof}
$P,Q\in \CC^{r\times n}$ being generic, by convention, is equivalent to choosing open dense $U_1,U_2\subseteq \CC^{r\times n}$. Since $\Upsilon$ and $\Upsilon_\CC$ are surjective (onto the Hermitian matrices in (ii)), and as algebraic maps continuous in the Zariski topology, the image of $U_1\times U_2$ (or $U_1\times U_1^*)$ will be open dense in the image as well.
\end{proof}

We now examine the signal and measurement varieties in more detail:
\begin{Prop}\label{Prop:irrvarieties}
Keep the notations of Section~\ref{sec:theorems.algpr}. For any $r\in\NN$, the varieties $\calP_\CC(r)$ and $\calP_\rho(r)$ are:
\begin{description}
\item[(i)] irreducible.
\item[(ii)] observable over the reals.
\item[(iii)] defined over the reals.
\end{description}
 In particular, this holds for $\calS_\CC=\calP_\CC(1)$ and $\calS_\rho=\calP_\rho(1)$ as well.
\end{Prop}
\begin{proof}
(i) For $\calP_\CC (r)$, irreducibility follows from surjectivity of $\Upsilon_\CC$, Proposition~\ref{Prop:irreducible} and irreducibility of complex affine space. Similarly, for $\calP_\rho (r)$, the statement follows from surjectivity of $\Upsilon$, and Proposition~\ref{Prop:irreducible}.\\
(ii) follows from considering the maps $\Upsilon_\CC$ and $\Upsilon$ over the reals, observing that the rank its Jacobian is not affected by this.\\
(iii) follows from (i), (ii) and Proposition~\ref{Prop:realdefobs}~(iii).\\
\end{proof}

\begin{Prop}\label{Prop:unrphi}
Keep the notations of Section~\ref{sec:theorems.signals} and~\ref{sec:theorems.measurem}. Assume that $\calS=\calS_\CC$ or $\calS_\rho$. Then $\phi_A$ is generically unramified for any $A\in\left(\left(\CC^{n\times n}\right)^\gamma\right)^k$. Furthermore, if $\calP$ contains $\calS$ (that is, all rank one signals), then $\phi$ is generically unramified.
\end{Prop}
\begin{proof}
$\calS$ and $\calP^{(k)}\times \calS$ are irreducible by Proposition~\ref{Prop:irrvarieties}. By Proposition~\ref{Prop:opencert}, it therefore suffices to show that there exists $x$ in the image of $\phi_A$ or $\phi$ such that $x$ does not ramify - but a generic choice of signal and/or measurement will suffice.
\end{proof}

\subsection{Proofs of Main Theorems}\label{app:proofs}
This section contains the technical proofs for our main theorems, which are stated in a slightly longer version.

\begin{Thm}\label{Thm:algsignal}
For a fixed measurement regime $(A_1,\dots, A_k),$ consider the three cases
\begin{description}
   \item[(a)] A generic signal $Z\in\calS$ is not identifiable from $\phi_A(Z)$.
   \item[(b)] A generic, but not all signals $Z\in\calS$, are identifiable from $\phi_A(Z)$.
   \item[(c)] All signals $Z\in\calS$ are identifiable from $\phi_A(Z)$.
\end{description}
The three cases above are equivalent to
\begin{description}
   \item[(a)] No signal $Z\in\calS$ is perturbation-stably identifiable from $\phi_A(Z)$.
   \item[(b)] A generic, but not all signals $Z\in\calS$, are perturbation-stably identifiable from $\phi_A(Z)$.
   \item[(c)] All signals $Z\in\calS$ are perturbation-stably identifiable from $\phi_A(Z)$.
\end{description}
Any triple of cases above is furthermore equivalent to
\begin{description}
   \item[(a)] $\phi_A$ is not birational.
   \item[(b)] $\phi_A$ is birational, but not an isomorphism.
   \item[(c)] $\phi_A$ is an isomorphism.
\end{description}
In particular, the three cases, in either of the three formulations, are mutually exclusive and exhaustive.
\end{Thm}
\begin{proof}
Mutual exclusivity and exhaustiveness of (a),(b),(c) follow from the third, algebraic formulation and elementary logic, once equivalence is established.

We prove equivalence of the first and second triple. Equivalence of (c) in the first and second triple follows from the fact that if all signals are identifiable, then all signals are perturbation-stably identifiable, since $\calS$ is an open neighborhood of any signal $Z\in\calS$. The converse follows from the fact that perturbation-stably identifiable signals are identifiable. Equivalence of (a) and (b) the first and second triple then follows from the assertion in Proposition~\ref{Prop:algcritsignal} that the perturbation-stably identifiable signals form a Zariski open subset of $\calS$, and the perturbation-stable signals are a subset of the identifiable signals.

We will now prove equivalence of the second and third triple. For that, note that if $\phi_A$ is birational if and only if there is $Z\in\calS$ with $\# \phi_A^{-1}\phi_A(Z) = 1,$ and an isomorphism if and only if there is no $Z\in\calS$ with $\# \phi_A^{-1}\phi_A(Z) \neq 1.$ Proposition~\ref{Prop:algcritsignal} then establishes the equivalence of the second and third triple.
\end{proof}

\begin{Thm}\label{Thm:algmeasure}
Assume that $\phi$ is generically unramified. Consider the three cases
\begin{description}
   \item[(a)] A generic measurement regime $A\in\calP^{k}$ is non-identifying.
   \item[(b)] A generic measurement regime $A\in\calP^{k}$ is incompletely identifying.
   \item[(c)] A generic measurement regime $A\in\calP^{k}$ is completely identifying.
\end{description}
The three cases above are equivalent to
\begin{description}
   \item[(a)] A generic measurement regime $A\in\calP^{k}$ is stably non-identifying. No measurement regime $A\in\calP^{k}$ is stably generically identifying.
   \item[(b)] A generic measurement regime $A\in\calP^{k}$ is stably incompletely identifying.
   \item[(c)] A generic measurement regime $A\in\calP^{k}$ is stably completely identifying.
\end{description}
Any triple of cases above is furthermore equivalent to
\begin{description}
   \item[(a)] $\phi$ is not birational.
   \item[(b)] $\phi$ is birational, and there is no open dense $U\subseteq \calP^{k}$ such that $\phi$ is an isomorphism on $U\times \calS$.
   \item[(c)] $\phi$ is birational, and there is an open dense $U\subseteq \calP^{k}$ such that $\phi$ is an isomorphism on $U\times \calS$.
\end{description}
In particular, the three cases, in either of the three formulations, are mutually exclusive and exhaustive.
\end{Thm}
\begin{proof}
The proof is analogous to that of Theorem~\ref{Thm:algsignal}.
\end{proof} 

\end{document}